\documentclass[nohyperref]{article}

\usepackage{microtype}
\usepackage{graphicx}
\usepackage{subfigure}

\usepackage{hyperref}

\usepackage[preprint]{neurips_2023}

\usepackage{amsmath}
\usepackage{amssymb}
\usepackage{mathtools}
\usepackage{amsthm}

\usepackage[capitalize,noabbrev]{cleveref}

\theoremstyle{plain}
\newtheorem{theorem}{Theorem}[section]
\newtheorem{proposition}[theorem]{Proposition}
\newtheorem{lemma}[theorem]{Lemma}
\newtheorem{corollary}[theorem]{Corollary}
\theoremstyle{definition}

\newtheorem{assumption}[theorem]{Assumption}
\theoremstyle{remark}

\usepackage{subfigure}
\usepackage{hyperref}
\hypersetup{colorlinks,linkcolor={blue},citecolor={blue},urlcolor={red}}  
\usepackage{natbib}
\usepackage{url}
\usepackage{amsthm}
\usepackage{amssymb}

\theoremstyle{aside}

\usepackage{color}

\def\1{\mathbf{1}}

\DeclareMathOperator*{\argmax}{arg\,max}

\def\la{\langle}
\def\ra{\rangle}

\newcommand{\dv}[1]{{\textcolor{orange}{#1}}}

\def\munderline#1{\underline{\sbox\tw@{$#1$}\dp\tw@\z@\box\tw@}}

\def\re{\mathbb{R}}
\def\a{\alpha}
\def\g{\gamma}
\def\dist{{\rm dist}}
\def\bd{m}

\usepackage[textsize=tiny]{todonotes}

\title{Adaptive Methods for Variational Inequalities under Relaxed Smoothness Assumption}

\author{%
  Daniil Vankov$^{1}$ \quad
  Angelia Nedi\'{c}$^{1}$ \quad
  Lalitha Sankar$^{1}$ \\
$^{1}$Arizona State University \\
\{dvankov, angelia.nedich, lsankar\}@asu.edu
}

\begin{document}

\maketitle

\begin{abstract}
  Variational Inequality (VI) problems have attracted great interest in the machine learning (ML) community due to their application in adversarial and multi-agent training. Despite its relevance in ML, the oft-used strong-monotonicity and Lipschitz continuity assumptions on VI problems are restrictive and do not hold in practice. To address this, we relax smoothness and monotonicity assumptions and study structured non-monotone generalized smoothness. The key idea of our results is in adaptive stepsizes. We prove the first-known convergence results for solving generalized smooth VIs for the three popular methods, namely, projection, Korpelevich, and Popov methods. Our convergence rate results for generalized smooth VIs match or improve existing results on smooth VIs. We present numerical experiments that support our theoretical guarantees and highlight the efficiency of proposed adaptive stepsizes.
\end{abstract}

\section{Introduction}
We consider a constrained variational inequality (VI) problem on a set $U \subseteq \mathbb{R}^m$. Given an operator $F: U \rightarrow \mathbb{R}^m$, the variational inequality problem, denoted as VI($U, F$), is defined as  
\begin{equation}
\begin{aligned}
\label{VI}
\hbox{find $u^* \in U$ such that } \langle F(u^*), u - u^* \rangle \geq  0, \; \forall u \in U. 
\end{aligned}
\end{equation}
Many numerical problems, such as constrained optimization, saddle-point problems, and multi-agent games are important and pertinent practical examples of VI problems. The recent interest in VIs is due to their application in machine learning applications such as generative adversarial networks (GANs) (\cite{gemp2018global, DBLP:conf/iclr/GidelBVVL19}) and reinforcement learning (\cite{daskalakis2020independent, kotsalis2022simple}).

The VI problem has been extensively studied (\cite{facchinei2003finite, tseng1995linear,  beznosikov2022smooth}) under (strong)-monotonicity and Lipschitz continuity assumptions on the operator. Two optimal solution approaches for solving VIs under these assumptions are the well-known extragradient (which we interchangeably refer to as the Korpelevich method) (\cite{ korpelevich1976extragradient}), and optimistic gradient (which we refer to as the Popov method here)~(\cite{popov1980modification}) methods; both require taking an extra-gradient step in every iteration though Popov requires only one oracle call to do so. Despite this progress, machine learning (ML) applications do not satisfy such assumptions (\cite{diakonikolas2021efficient, DBLP:conf/iclr/ZhangHSJ20}). To this end, 
recent works (\cite{DBLP:conf/iclr/WeiLZL21, diakonikolas2021efficient, vankov2023last}) have proposed and investigated new classes of structured non-monotone operators to address non-monotonicity in VIs. Furthermore, while the bulk of research on VIs assume Lipschitz continuity or boundedness of the operator, more recently,  guarantees for non-smooth VIs with only a linear growth assumption have been considered using Korpelevich (\cite{vlatakis2023stochastic}) and Popov (\cite{vankov2023last}) methods.

For the class of smooth VI operators, 
{\cite{malitsky2020golden} propose adaptive  golden ratios stepsizes to adapt to a local Lipschitz constant.  \cite{DBLP:journals/tmlr/Bohm23} utilize this adaptive schedule for the Extragradient method and prove sublinear convergence under a weak Minty assumption}. \cite{gorbunov2022clipped} use clipping techniques to show convergence of stochastic Extragradient method under heavy tail noise. They use the fact that the operator is continuous to show that iterates of the \emph{clipped} Korpelevich method stay within the bounded set. However, the resulting Lipschitz constants can be very large, thus leading to extremely small stepsizes and slow convergence rates. 
More generally, for non-smooth VIs, despite the lack of theoretical results, adaptive and normalized first-order methods are popular and efficient in practice for training GANs (\cite{DBLP:conf/iclr/BrockDS19}, \cite{karras2020analyzing}). 
Recently, \cite{jelassi2022dissecting} argues that the success of the ADAM optimizer~(\cite{DBLP:journals/corr/KingmaB14}) for GANs is due to normalization. 


\begin{table*}[t]
    \begin{center}
    \label{tab:table1}
    \begin{tabular}{|l|l|l|l|}
        \hline
       \textbf{Methods} &$p=2$ & $p > 2$ \\
      \hline
      Projection & Asym (Thm \ref{thm-projection-asym}) & Asym  (Thm \ref{thm-projection-asym}) \\
     \hline
     Korpelevich  & $D_0^2 \exp [- \mu \underline{\gamma} k ] $  ( Thm \ref{thm-Korp-rates})  &  $ D_0^2(1 + \mu \underline{\gamma}   (D_0^2)^{(p-2)/2} k)^{-2/(p-2)}$  ( Thm \ref{thm-Korp-rates})\\ 
     \hline
     Popov &  $R_1^2 \exp [- \mu \underline{\gamma} k ]$  (Thm \ref{thm-Popov-a01-rates})  & $ R_1^2 (1 + p C   (R_1^2)^{(p-2)/2} k)^{-2/(p-2)}$ (Thm \ref{thm-Popov-a01-rates}) \\
     \hline

    \end{tabular}
    \end{center}
        \caption{ 
     Summary of results on rates at which the quantities $R_k^2$ and $D_k^2$ decrease with the number $k$ of iterations. We use ``Asym" as an abbreviation for asymptotic results.  Performance measure for Popov method is $R_k^2 = \dist^2(u_k, U^*) + \|u_k - h_{k-1}\|^2$, and for Korpelevich method is $D_k^2 = \dist^2(h_k, U^*)$. 
    Constant \underline{$\gamma$} is a lower bound of adaptive stepsizes, i.e. $\underline{\gamma} \leq \gamma_k$ for all $k \geq 0$. Constant $C = 2^{-p/2} 
 \min \{ \frac{1}{2}(2 R_1^2)^{-(p-2)/2} , 2^{-p+2} \mu \underline{\gamma}\}$. 
 }
\end{table*}

Non-smoothness assumptions on the operator are of increasing interest across the broader optimization literature in the context of ML. In the context of deep learning optimization~\cite{DBLP:conf/iclr/ZhangHSJ20}
observed a linear dependence between the norm of the operator and its Jacobian. Based on this observation, they introduced the $(L_0, L_1)$-smoothness assumption, under which the sublinear convergence rate for the clipped gradient method is derived. 
Following this,  \cite{DBLP:conf/icml/00020LL23} proposed an $\a$-symmetric class of \emph{generalized smooth operators} that allows extending $(L_0, L_1)$-smoothness to non-differentiable operators. For this class, the sublinear convergence rate for normalized gradient descent was proved.

We introduce the class of structured non-monotone \emph{generalized smooth VIs} with $\alpha$-symmetric operators.
We address the following two natural questions in this context and answer them both in the affirmative: (i) Can convergence guarantees be assured for VIs under the assumption of generalized smoothness of the operator?; and (ii) If so, is the convergence for generalized smooth VIs the same as that for smooth VIs?

We focus on the class of structured non-monotone VIs with $p$-quasi sharp operators under the weakest known assumption on generalized smoothness called $\a$-symmetricity. The class of $p$-quasi sharp operators includes weak sharp and quasi-strongly monotone operators and coincides with the Saddle-Point Metric Subregularity (SP-MS) assumption for $p \geq 2$. Analogously, the class of $\a$-symmetric operators is the weakest class of continuous operators and includes Lipschitz continuous and $(L_0, L_1)$-smooth operators.  Our key contributions are summarized below (see also Table~\ref{tab:table1})
\begin{itemize}
    \item We provide the first known analysis of first-order methods for solving structured non-monotone VIs under generalized smooth assumption.
    In particular, we focus on a class of $p$-quasi sharp and $\a$-symmetric operators. 
    The key feature of our analysis is the use of cleverly chosen adaptive stepsizes for three well-known methods, namely, projection, Korpelevich, and Popov. 
    \item  For the projection method with adaptive stepsizes, we prove asymptotic convergence to a solution. Moreover, as a consequence, we show that the method converges linearly to the $1$-neighborhood of the solution.
    \item For Korpelevich and Popov methods, we prove asymptotic convergence to a solution. Using the $\a$-symmetricity assumption, we show that the sequence of provided adaptive stepsizes are lower bounded for both methods. Moreover, we show a linear convergence rate for both methods for $2$-quasi sharp operators. For $p>2$, we show a rate of $O(k^{-2/p})$, which is faster than the rate $O(k^{-1/(p-1)})$ provided in \cite{DBLP:conf/iclr/WeiLZL21} under more restrictive assumptions on monotonicity and Lipschitz continuity of the operator.
    \item When the parameters of the problem such as $p, \alpha,L_0,L_1$ are unknown or overestimated, for the Korpelevich method, we introduce an easier method to fine-tune stepsize parameters. We do so by separating the clipping part from the stepsize parameter sequence.
    For these easy-to-tune stepsizes, we prove asymptotic convergence. In particular, we introduce a \emph{descent inequality}, which is crucial to obtain convergence rates under different step-size schedules. 
    \item Finally, we present numerical experiments to compare the performance of the three considered methods with proposed adaptive stepsizes and Korpelevich with golden ratio stepsizes~(\cite{DBLP:journals/tmlr/Bohm23}) for different $\a$ and $p$. We observe that our proposed stepsizes outperform golden ratio stepsizes, and as predicted by our theory, the convergence of all methods slows down as $\alpha$ and $p$ increase. We also highlight an experiment with decreasing adaptive stepsizes that are independent of the problem parameters. The key takeaway from this experiment is the observation that projection appears to significantly slow down relative to the other two methods for various parameter settings. 
\end{itemize}

\section{Assumptions on VI}


The operator $F$ is said to be Lipschitz continuous if there exists $L > 0$ such that:
\begin{equation}
\label{def-Lipschitz}
\|F(u) - F(v)\| \leq L \|u - v\| \quad \hbox{for all } u,v \in U.
\end{equation}
While being the most common assumption used in the literature, it is quite restrictive. Recently, a weaker assumption has been  considered in \cite{DBLP:conf/iclr/ZhangHSJ20}, given as follows:
\begin{equation}
\label{def-generalized-smooth}
\|\nabla F(u)\| \leq L_0 + L_1 \|F(u)\| \quad \text{ for all } u \in U.
\end{equation}
When the operator $F(\cdot)$ is differentiable and $L$-Lipshitz continuous, it satisfies~\eqref{def-generalized-smooth} with $L_0 = L$ and $L_1 = 0$.
In the experiments reported in \cite{DBLP:conf/iclr/ZhangHSJ20}, it was observed that the neural network models tend to satisfy the condition in~\eqref{def-generalized-smooth}.
Later in~\cite{DBLP:conf/icml/00020LL23}, a weaker smoothness condition than~\eqref{def-generalized-smooth} was introduced leading to  a class of operators termed {\it $\a$-symmetric}, which includes operators $F(\cdot)$ satisfying the following relation for some constants $L_0, L_1 \geq 0$ and for all $u, v \in U$:
\begin{equation}
\label{def-general-alpha}
\|F(u) - F(v)\| \leq 
\left (L_0 + L_1 \max_{\theta \in (0,1)}\|F(w_{\theta})\|^{\alpha}\right) \|u - v\|,
\end{equation}
where $w_{\theta} = \theta u + (1 - \theta) v$ and $\a\in(0,1]$.
If a differentiable operator satisfies
the condition in~\eqref{def-generalized-smooth}, then it satisfies the condition in~\eqref{def-general-alpha}. Moreover, it has been shown in \cite{DBLP:conf/icml/00020LL23} that the class of differentiable and $\a$-symmetric operators with $\a=1$ is equivalent to the class of operators with linear growth of the Jacobian, as in~\eqref{def-generalized-smooth}. Since the class of $\a$-symmetric operators is a rather wide class, we focus on such operators.

\begin{assumption}\label{asum-alpha} 
Given a set $U\subseteq\mathbb{R}^m$, the operator $F(\cdot):U\to\mathbb{R}^\bd$ is $\alpha$-symmetric over $U$, i.e.,
for some $\alpha\in(0,1]$ and $L_0, L_1 \geq 0$, we have for all $u, v \in U$,
\begin{equation}
\label{assum-general-alpha}
\|F(u) - F(v)\| \leq 
\left (L_0 + L_1 \max_{\theta \in (0,1)}\|F(w_{\theta})\|^{\alpha}\right) \|u - v\|,
\end{equation}
where $w_{\theta} = \theta u + (1 - \theta) v$.
\end{assumption}

An alternative characterization of an $\a$-symmetric operators has been proved in~\cite{DBLP:conf/icml/00020LL23}, as given in the following proposition.

\begin{proposition}[\cite{DBLP:conf/icml/00020LL23}, Proposition 1]
\label{prop-a} 
Let $U\subseteq \mathbb{R}^m$ be a nonempty set and let
$F(\cdot): U\to\re^m$ be an operator. Then, we have
\begin{itemize}
    \item [(a)] $F(\cdot)$ is $\a$-symmetric with $\a \in (0,1)$ and constants $L_0, L_1
\ge 0$ if and only if the following relation holds for all $y, y' \in U$,
\begin{equation}
\begin{aligned}
\label{alpha-property}
\|F(y) - F(y')\| &\leq \|y - y'\| (K_0 + K_1 \|F(y')\|^{\alpha} + K_2 \|y - y'\|^{\a / (1 - \a)}),
\end{aligned}
\end{equation}
where $K_0 = L_0(2^{\a^2 / (1 - \a)} + 1)$, $K_1 = L_1 2^{\a^2 / (1 - \a)} 3^\a$, and $K_2 = L_1^{1 / (1 - \a)} 2^{\a^2 / (1 - \a)} 3^{\a} (1 - \a)^{\a / (1 - \a)} $.

\item[(b)] 
$F(\cdot)$ is $\a$-symmetric with $\a =1 $ and constants $L_0, L_1\ge0$ if and only if the following relation holds for all $y, y' \in U$,
\begin{equation}
\begin{aligned}
\label{alpha-property-1}
&\|F(y) - F(y')\| \leq \|y - y'\| (L_0+ L_1 \|F(y')\|) \exp (L_1 \|y - y'\|).
\end{aligned}
\end{equation}
\end{itemize}
\end{proposition}

Proposition~\ref{prop-a} is useful for our analysis, since it  describes $\alpha$-symmetric operator by using two points $y, y' \in U$, and bypasses the evaluation of $\max_{\theta\in(0,1)}\|F(w_\theta)\|^\alpha$.


A number of structured non-monotonicity assumptions were proposed to make the assumptions on the operator closer to real-life problems. Such assumptions include strong coherency~(\cite{song2020optimistic}), weak Minty~(\cite{choudhury2023single}), Saddle-Point Metric Subregularity (SP-MS)~(\cite{DBLP:conf/iclr/WeiLZL21}), and $p$-quasi sharpness~(\cite{vankov2023last}). 
Each of these assumptions imposes a special structured lower bound for the quantity $\langle F(u), u - u^* \rangle$, where $u^*$ is a solution to the underlying VI problem. 
In the recent work (\cite{vankov2023last}), it has been shown that the class of $p$-quasi sharp operators includes strongly-monotone and strongly coherent operators,
and coincides with the class of operators satisfying SP-MS condition for $p>2$  which need not be monotone. Due to the generality of the $p$-quasi sharpness property, 
we consider the class of $p$-quasi sharp operators. To formally introduce this class, we define the solution set for the VI($U,F$), which is denoted by $U^*$, as follows:
\[U^*=\{u^*\in U\mid \langle F(u^*), u - u^* \rangle \geq  0 \ \ \hbox{for all } u \in U\}.\]

Throughout this paper, we make the following assumption on the constraint set and the solution set.

\begin{assumption}\label{asum-set}
The set $U\subseteq\mathbb{R}^m$ is a nonempty closed convex set, and the solution set $U^*\subseteq\mathbb{R}^m$ is nonempty and closed. 
\end{assumption}
\begin{assumption}\label{asum-sharp} The operator $F(\cdot):U\to\mathbb{R}^\bd$ has a $p$-quasi sharpness property over $U$ relative to the solution set $U^*$, i.e.,
for some $p>0$, $\mu >0$, and for all $u \in U$ and all $u^* \in U^*$,
\begin{equation}\label{eq-psharp}
\langle F(u), u - u^* \rangle \geq \mu \,\dist^p(u, U^*).
\end{equation}

\end{assumption}

Next, motivated by the example provided in \cite{vankov2023last},
we present an operator that  is $p$-quasi sharp and $\a$-symmetric.


\begin{proposition} \label{example1}
Let $F(u) = \begin{bmatrix} \text{sign}(u_1) |u_1|^{p - 1} + u_2 \\ \text{sign}(u_2) |u_2|^{p - 1} - u_1 \end{bmatrix}$ for $u\in\re^2$. Then, the operator $F(\cdot)$ is $p$-quasi sharp for $p \geq 1$ with $\mu = 2^{1-p}$. Moreover, when $p > 2$, the operator is $\a$-symmetric with $\alpha= \frac{p-2}{p-1}$, and  $L_0 = 1+ (p-1) 2^{1/2} 4^{1/(p-1)}, L_1 = 2(p-1) 2^{1/2(p-1)} $.
\end{proposition}

Note that for this example, the parameter $p$ of quasi-sharpness depends on parameter $\alpha$ of $\a$-symmetricity and vice versa. To illustrate the importance of $\a$-symmetric assumption consider the operator $F(u)$ from Proposition~\ref{example1} with $p=9.0$ and $\alpha = 0.875$ on a ball  $B(0, 10.0)$. Operator $F(u)$ is Lipschitz continuous in the set $B(0, 10.0)$, but the Lipschitz constant is larger than $L=10^7 - 1$, while $L_0 = 1 + 8 \times 2^{3/4}$ and $L_1 = 16 \times 2^{1/16}$ are relatively small.

\section{Methods with Adaptive Stepsizes}
Our goal is to develop methods and provide convergence guarantees for solving non-monotone VI problems with $\a$-symmetric operators. In this section, we consider three popular methods for solving VIs.
The first method of our interest is the projection method, given as follows: 
\begin{equation}
\label{eq-proj-det}
u_{k+1} =  P_{U}(u_k - \g_k F(u_k))
\qquad\hbox{for all $k \geq 0$},
\end{equation}
where $\g_k>0$ is the step-size and $u_0 \in U$ is an arbitrary 
initial point. 
Despite its simplicity and efficiency in optimization, the projection method may diverge for monotone VIs (\cite{DBLP:conf/iclr/DaskalakisISZ18}). 
The projection method converges to the solution for strongly monotone VIs but is not optimal. In fact, two methods proposed in \cite{korpelevich1976extragradient, popov1980modification}, named extragradient and optimistic methods are optimal for strongly monotone VIs. In the recent works (\cite{diakonikolas2021efficient, DBLP:conf/iclr/WeiLZL21, loizou2021stochastic, vankov2023last}) authors considered structured non-monotone VIs. The convergence rate of Popov and Korpelevich methods for quasi-strong monotone and $2$-quasi sharp operators matches optimal convergence rates (\cite{beznosikov2022smooth, gorbunov2022stochastic, choudhury2023single, vankov2023last}). Motivated by these results, we consider these two methods. 
The Korpelevich method,
is given as follows. For all $k \ge 0$,
\begin{equation}
\begin{aligned}
\label{eq-korpelevich-det}
u_{k} &=  P_{U}(h_k - \g_k F(h_k)), \cr
h_{k+1} &= P_{U}(h_{k} - \g_{k} F(u_{k})), 
\end{aligned}
\end{equation}
where $\g_k$ is step-size and $h_0 \in U$ is an arbitrary initial point.
%
The Popov method,
is presented below. For all $k \ge 0$,
\begin{equation}
\begin{aligned}
\label{eq-popov-det}
u_{k+1} &=  P_{U}(u_k - \g_k F(h_k)), \cr
h_{k+1} &= P_{U}(u_{k+1} - \g_{k+1} F(h_{k})), 
\end{aligned}
\end{equation}
where $\g_k>0$ is the step-size and $u_0, h_0 \in U$ are arbitrary initial points.
\section{Convergence Analysis for Projection Method}

In this section, we provide convergence results for projection method~\eqref{eq-proj-det}. To show convergence of projection method for $\a$-symmetric and $p$-quasi sharp operators we define adaptive stepsizes as:
\begin{equation}
\begin{aligned}
\label{proj-stepsizes}
    \gamma_k = \beta_k \min \left\{ 1, \frac{1}{\|F(u_k)\|} \right\}. \cr
\end{aligned}
\end{equation}

\begin{theorem}
    \label{thm-projection-asym}
Let Assumptions~\ref{asum-alpha}, \ref{asum-set}, \ref{asum-sharp}  hold. Also, let stepsizes by given by~\eqref{proj-stepsizes} and sequence $\beta_k$ be such that  $\sum_{k=0}^{\infty}\beta_k = \infty$,$\sum_{k=0}^{\infty}\beta_k^2 < \infty$.
Then,  the following statements hold for the iterate sequences $\{u_k\}$  generated by the projection method~\eqref{eq-proj-det}:

(a) The following inequality holds for any $k\geq 0$, 
\begin{equation}
\begin{aligned}
\label{eq-projection-thm-asym}
\|u_{k+1}-u^*\|^2 &\leq  \|u_k -u^*\|^2 - 2\g_k \dist^p(u_k, U^*) + \beta_k^2 
\end{aligned}
\end{equation}
Moreover, for any solution $u^* \in U^*$ the sequence $\{\|u_k - u^*\| \}_{k=0}^{\infty}$ is bounded by a constant $B(u^*)$.

(b) The stepsizes sequence $\{\gamma_k\}$ is bounded by
\[\gamma_k \ge \beta_k \min \left\{1, \frac{1}{C_1}, \frac{1}{\overline{C}_1} \right\}\qquad\hbox{for all } k\ge0,\]
where $C_1$ and $\bar C_1$ are given by
\begin{equation*}
\begin{aligned}
C_1 &= B(u_0^*) (K_0 + K_1 \|F(u^*_0)\|^{\alpha} + K_2 B(u_0^*)^{\a / (1 - \a)}) + \|F(u^*_0)\|
\end{aligned}
\end{equation*}
\begin{equation*}
\begin{aligned}
\overline{C}_1 &= B(u_0^*) (L_0+ L_1 \|F(u^*_0)\|) \exp (L_1 B(u_0^*)) + \|F(u^*_0)\|,
\end{aligned}
\end{equation*}
and $u^*_0 = P_{U^*}(u_0)$.

(c) The sequence $\{u_k\}$ converges to some solution $\bar u \in U^*$. 

\end{theorem}

Interestingly, to achieve the descent inequality~\eqref{eq-projection-thm-asym} we don't use Assumption~\ref{asum-alpha}. But this assumption plays a crucial role in showing the upper bound on a sequence $\{\|F(u_k)\|\}$, and as a result a lower bound on the stepsizes. In order to get converges from equation~\eqref{eq-projection-thm-asym}, it is required to show that stepsize sequence is square summable, i.e. $\sum_{k=0}^{\infty} \gamma_k^2 \leq \infty$. 

\begin{corollary}
\label{cor-proj}
The projection method with adaptive stepsizes as given in~\eqref{proj-stepsizes} and Theorem~\ref{thm-projection-asym} enjoys a linear convergence rate to a $1$-neighborhood of the solution set for $p \geq 2$.
\end{corollary}
Corollary~\ref{cor-proj} follows straight from equation~\eqref{eq-projection-thm-asym} in Theorem~\ref{thm-projection-asym} (a) and stepsizes from \cite{stich2019unified}. In Section~\ref{Numerical}, we show that this rate holds in our experiments.
Moreover, results of Theorem~\ref{thm-projection-asym} hold for normalized projection method~\eqref{eq-proj-norm-det}, that can be viewed as deterministic version of normalized SGD (\cite{jelassi2022dissecting}):
\begin{equation*}
\label{eq-proj-norm-det}
u_{k+1} =  P_{U}(u_k - \beta_k \frac{1}{\|F(u_k)\|} F(u_k)),
\; \forall \; k \geq 0.
\end{equation*}
\section{Convergence Analysis for Korpelevich Method}
In this section, we consider the Korpelevich method with adaptive stepsizes. We show the asymptotic convergence of the method and provide its convergence rate. Firstly, we derive a basic inequality for Korpelevich method~\eqref{eq-korpelevich-det} with arbitrary stepsizes without any assumptions on the operator.
This lemma is the basis for all the subsequent results.
\begin{lemma} \label{Lemma1-Korp} 
Let $U$ be a nonempty closed convex set. Then, for the iterate sequences
$\{u_k\}$ and $\{h_k\}$ generated by the Korpelevich method~\eqref{eq-korpelevich-det} we have
for all $y \in U$ and $k \ge 0$,
\begin{equation}
\begin{aligned}
\label{eq-lemma-korp}
 \|h_{k+1}  - y\|^2  \leq &
 \|h_k  - y\|^2 - \|h_{k} -u_k\|^2 - 2 \g_{k} \langle F(u_k), u_k - y\rangle \cr
 &+ \g_{k}^2 \,\|F(h_k) -  F(u_k) \|^2.
\end{aligned}
\end{equation}
\end{lemma}

The proof of Lemma~\ref{Lemma1-Korp} can be found in Appendix~\ref{App-Korp}.
Next we provide adaptive step-sizes for Korpelevich method
employed with an $\a$-symmetric operator with $\a \in (0,1]$.
When $\a\in(0,1)$, the stepsizes are defined as follows: for all  $k\ge0$,
\begin{align}
\label{korp-step-01}
\g_k = \min &\left\{\frac{1}{4\mu} ,\frac{1}{3 \sqrt{2} K_0}, \frac{1}{\|F(h_{k})\|},  \frac{1}{ 3 \sqrt{2} K_1 \|F(h_{k})\|^{\alpha}}, 
\frac{1}{ 3 \sqrt{2} K_2 } \right\}, \quad  
\end{align}
while for $\a = 1$, the  stepsizes are given by:
for all  $k\ge0$,
\begin{align}
\label{korp-step-1}
\g_k = \min &\left\{ \frac{1}{4\mu} , \frac{1}{2 \sqrt{2} e L_0},  \frac{1}{2 \sqrt{2} e L_1 \|F(h_{k})\| } \right\} .
\end{align} 
The constants $L_0>0$ and $L_1>0$ are related to the $\a$-smoothness of the operator, while $K_0, K_1, K_2>0$ are defined in terms of $L_0$ and $L_1$ as given in Proposition~\ref{prop-a}.

We present a theorem on the convergence of the Korpelevich method \eqref{eq-korpelevich-det} with adaptive stepsizes~\eqref{korp-step-01}, \eqref{korp-step-1},
to the solution set
for $p$-quasi sharp and $\alpha$-symmetric operators with $\a \in (0,1]$. There are two key ideas behind the proof. 

Firstly, one need to upper bound $\gamma_k^2 \|F(h_{k}) - F(u_k)\|^2$ from Lemma~\ref{Lemma1-Korp} to get descent inequality. To do that, we use Proposition~\ref{prop-a} of $\a$-symmetric operator and carefully chosen adaptive stepsizes $\{\gamma_k\}$ ~\eqref{korp-step-01}, \eqref{korp-step-1}. The stepsizes utilize the idea of normalization of operator or operator clipping. Notice, that based on Proposition~\ref{prop-a}, $\|F(h_{k}) - F(h_{k-1})\|$ depends on $\|F(h_k)\|$ and $\|u_k - h_k\|$, while based on the structure of the method quantity  $\|u_k - h_k\|$ can be upper bounded by $\gamma_k \|F(h_{k})\|$. This makes the stepsizes $\gamma_k$ depended only on operator value at point $h_k$ and parameters from Assumption~\ref{asum-alpha}.

Secondly, to obtain convergence guarantees it is crucial to find a lower bound for the stepsizes sequence. Such a lower bound is possible to find when an operator is $p$-quasi sharp, or more generally when $\la F(h_k), h_k - u^*\ra$ is positive. When this property holds, it is easy to observe that quantity $\|h_{k} - u^*\|^2$ is non-increasing for all $u^* \in U^*$ and $k\geq 0$. Using this fact and assumption on $\a$-symmetric operator we are able to present a lower bound on the adaptive stepsizes~\eqref{korp-step-01}, \eqref{korp-step-1}. Finally, we apply deterministic version of Robbins-Siegmund lemma (\cite{robbins1971convergence}, see Lemma \ref{lemma-polyak11-det} in Appendix) and get the desired convergence results. 
Formally, the theorem summarizing the result on convergence and lower bound on stepsizes is given below.

\begin{theorem}
\label{thm-Korpelevich-asym} 
Let Assumptions~\ref{asum-alpha}, \ref{asum-set}, \ref{asum-sharp}  hold. Also, let the stepsizes be given by \eqref{korp-step-01} if $\a \in (0,1)$  and by \eqref{korp-step-1} if $\a=1$.
Then,  the following statements hold for the iterate sequences $\{u_k\}$ and $\{h_k\}$ generated by the Korpelevich method~\eqref{eq-korpelevich-det}:

(a) The following descent inequality holds for any solution $u^* \in U^*$
and any $k\geq 0$, 
\begin{equation}
\begin{aligned}
\label{eq-korpelevich-thm-asym}
 \|h_{k+1}  - u^*\|^2 &\leq \|h_k - u^*\|^2  -\frac{1}{2}\|u_k - h_k\|^2 - 2 \g_k \mu \dist^p( u_k, U^*).
\end{aligned}
\end{equation}
(b) The stepsizes sequence $\{\gamma_k\}$ is bounded below, i.e.,
\[\gamma_k\ge \underline{\gamma}\qquad\hbox{for all } k\ge0.\]
For $\alpha\in(0,1)$, the constant $\underline{\gamma}$ is given by
\begin{align*}
\underline{\gamma} = \min \left\{ \frac{1}{4\mu} ,\frac{1}{3 \sqrt{2} K_0}, \frac{1}{C_1}, \frac{1}{3 \sqrt{2} K_1 C_1^{\alpha}}, 
 \frac{1}{3 \sqrt{2} K_2} \right\},
\end{align*}
and for $\alpha=1$, it is given by
\[\underline{\gamma}= \min \left\{ \frac{1}{4\mu}, \frac{1}{2 \sqrt{2} e L_0}, \frac{1}{2 \sqrt{2} e L_1 \overline{C}_1 } \right\}, \]
where the constants $K_0, K_1, K_2>0$ are as given in Proposition~\ref{prop-a}, while $C_1$ and $\bar C_1$ are given by
\begin{equation*}
\begin{aligned}
C_1 =
&D_0 (K_0 + K_1 \|F(P_{U^*}(h_0))\|^{\alpha} + K_2 (D_0)^{ \frac{\a}{ (1 - \a)}}) + \|F(P_{U^*}(h_0))\|,
\end{aligned}
\end{equation*}
\begin{equation*}
\begin{aligned}
\overline{C}_1 =
&D_0 (L_0+ L_1 \|F(P_{U^*}(h_0))\|) \exp (L_1 D_0) + \|F(P_{U^*}(h_0))\|,
\end{aligned}
\end{equation*}
and $D_0^2 = \dist^2(h_0, U^*)$.

(c) The iterates $u_k$ and $h_k$ converge to a solution $\bar{u} \in U^*$.

\end{theorem}
The proof of these results can be found in Appendix~\ref{App-Korp-asym}. 
The conditions on a stepsize sequence $\{\gamma_k\}$ for Korpelevich method in Theorem~\ref{thm-Korpelevich-asym} are less restrictive than conditions on $\{\gamma_k\}$ for projection method from Theorem~\ref{thm-projection-asym}. In Theorem~\ref{thm-projection-asym}, stepsizes should be diminishing, while for Korpelevich method, it is possible to have stepsizes with a global nonzero lower bound \underline{$\gamma$}.
Next, we provide convergence rate for Korpelevich method.

\begin{theorem}
\label{thm-Korp-rates}
Let Assumptions~\ref{asum-alpha}, \ref{asum-set}, \ref{asum-sharp}  hold. Consider the  Korpelevich method \eqref{eq-korpelevich-det} with the stepsizes $\g_k$ given by~\eqref{korp-step-01} for $\a \in (0,1)$, and by~\eqref{korp-step-1} for $\a = 1$, and let 
\[D^2_{k} = \dist^2(h_k, U^*)\qquad\hbox{for all }k\ge0.\] 
Then, the following results hold for the quantity $D^2_{k}$:\\
\textbf{Case (a)} If the operator $F(\cdot)$ is $p$-quasi-sharp with $p=2$, then  
\begin{equation}
    \begin{aligned}
    \label{eq-korp-basic-12}
     D^2_{k+1} & \leq (1 -  \mu \underline{\gamma}) D^2_k \qquad\hbox{for  all $k \ge 0$.}
    \end{aligned}
\end{equation}

\textbf{Case (b)}
If the operator $F(\cdot)$ is $p$-quasi-sharp with $p>2$, then 
\begin{equation}
    \begin{aligned}
    \label{eq-popov-basic-12}
     D^2_{k+1} & \leq \frac{D_0^2}{ (1 +  \mu \underline{\gamma} (k+1) (D_0^2)^{p/2})^{2/p}}\quad\hbox{for  all $k \ge 0$.}
    \end{aligned}
\end{equation}
The constant $\underline{\gamma}$ in~\eqref{eq-korp-basic-12} and~\eqref{eq-popov-basic-12} is the lower-bound on the stepsizes provided in Theorem~\ref{thm-Korpelevich-asym}. 

\end{theorem}
The proof of these results can be found in Appendix~\ref{App-Korp-rates}. Notice, that when $F$ is $L_0$-Lipschitz continuous, the constant $L_1=0$, and $\a=1$, then the stepsizes will be simplified to:
\begin{align}
\gamma_k = \min \left\{\frac{1}{4\mu}, \frac{1}{2\sqrt{2} e L_0}\right\}.
\end{align}
In this case, the rate provided Theorem~\ref{thm-Korp-rates}(a) recovers the optimal rate of $\mathcal{O}(\frac{L_0}{\mu} \exp (- \frac{\mu}{L_0}k))$ for $2$-quasi-sharp operators.
Notice, that for $p>2$, from equation~\eqref{eq-korpelevich-thm-asym} in Theorem~\ref{thm-Korpelevich-asym} (a), it follows that convergence to $1$-neighborhood of a solution is linear.


\vspace{-0.1in}
\section{Convergence Analysis for Popov method}
We provide convergence results for Popov method~\eqref{eq-popov-det} for solving structured non-monotone non-Lipschitz VIs. We start with a lemma on the iterates of Popov method which serve as a base for further convergence results.  We collate all our proofs in Appendix~\ref{App-Popov-asym-01}.
\begin{lemma} \label{Lemma-Popov} 
Let Assumption~\ref{asum-set} hold.
Then, for the iterate sequences
$\{u_k\}$ and $\{h_k\}$ generated by the Popov method~\eqref{eq-popov-det} we have
for all $y \in U$ and $k \ge 1$,
\begin{equation*}
\begin{aligned}
 \|u_{k+1}  - y\|^2 & \leq \|u_k  - y\|^2 - \|u_{k+1} -h_k\|^2 - \|u_k - h_k\|^2 \cr
 &- 2 \g_{k} \langle F(h_k), h_k - y\rangle + 2 \g_{k}^2 \,\|F(h_k) -  F(h_{k-1}) \|^2.
\end{aligned}
\end{equation*}
\end{lemma}

The stepsizes for Popov method for $\a$-symmetric operators with different $\a \in (0,1]$ are given below. When $\a \in (0,1)$ the stepsizes for all $k \geq 1$ are:
\begin{equation}
\begin{aligned}
\label{pop-step-01}
    \g_k = &\min 
    \left\{ \frac{1}{\|F(h_{k-1})\|}, \frac{1}{6\sqrt{2} K_0}, \frac{1}{6 \sqrt{2} K_1 \|F(h_{k-1})\|^{\a }}, \frac{1}{6 \sqrt{2} K_2 (\|u_k - h_{k-1}\| + 1)^{\a /(1 - \a)}}, \frac{1}{4\mu} \right\}.
\end{aligned}
\end{equation}
For $\a = 1$, the stepsizes for Popov method  for all $k \geq 1$ are
\begin{align}
\label{pop-step-1}
&\g_k  = \min \left\{ \frac{1}{4\mu} , \frac{1}{\|F(h_{k-1})\|}, 
 \frac{1}{2 \sqrt{2} (L_0 \!+\! L_1 \|F(h_{k-1})\|) \exp (L_1 \|u_k \!-\! h_{k-1}\| \!+\! 1)} \! \right\}.\ \ 
\end{align}
To prove the convergence of Popov method we use the same two ideas as in the proof of Korpelevich method: (i) provide an upper bound on  $\gamma_k^2 \|F(h_{k}) - F(h_{k-1})\|^2$ utilizing the adaptive stepsizes~\eqref{pop-step-01}, and ~\eqref{pop-step-1}; (ii) provide a lower bound on the stepsizes sequence. 
Based on  Proposition~\ref{prop-a}, the term $2 \g_{k}^2 \,\|F(h_k) -  F(h_{k-1}) \|^2$ from Lemma~\ref{Lemma-Popov} depends on $\|F(h_{k-1})\|$ and $\|h_{k} - h_{k-1}\|$. Unlike the bound in the analysis of Korpelevich method, in Popov method, the quantity $\|h_{k} - h_{k-1}\|$ can't be upper bounded by $\|F(h_{k-1})\|$. To address this issue, using the projection inequality and the method update rule, we provide a bound of $\|h_{k} - h_{k-1}\|$ using $\gamma_k\|F(h_{k-1})\|$ and $\|u_k - h_{k-1}\|$. Notice that, based on the method updates, the stepsize $\gamma_k$ can't depend on $h_k$. Thus, normalization by $\|F(h_k)\|$, $\|F(h_k) -  F(h_{k-1})\|$ or $\|h_{k} - h_{k-1}\|$ is not possible. Thus, $\g_k$ uses normalization by $\|u_k - h_{k-1}\|$ and $\|F(h_{k-1})\|$.
We next prove asymptotic convergence of Popov method with adaptive stepsizes.
\begin{theorem}
\label{thm-Popov-a01-asymp}
Let Assumptions~\ref{asum-alpha}, \ref{asum-set}, \ref{asum-sharp}  hold. Also, let the stepsizes be given by \eqref{pop-step-01} if $\a \in (0,1)$  and by \eqref{pop-step-1} if $\a=1$.
Then, the following statements hold for the iterate sequences $\{u_k\}$ and $\{h_k\}$ generated by the Popov method:

(a) The descent inequality holds for arbitrary solution $u^* \in U^*$ and any $k\geq 1$
\begin{equation}
    \begin{aligned}
    \label{eq-th-pop-a01}
     \|u_{k+1}  - u^*\|^2 + \|u_{k+1} -h_k\|^2   &\leq \|u_k  - u^*\|^2   + \frac{1}{2}\|u_k - h_{k-1}\|^2 \cr
     &- \frac{1}{2} \|u_k - h_k\|^2 - 2 \g_{k} \langle F(h_k), h_k - u^*\rangle.
    \end{aligned}
\end{equation}
(b) The stepsizes sequence $\{\gamma_k\}$ is bounded below, i.e.,
\begin{align*}
\gamma_k\ge \underline{\gamma}\qquad\hbox{for all } k\ge1.
\end{align*}
For $\alpha\in(0,1)$, the constant $\underline{\gamma}$ is given by
\begin{align} \label{eq-th-pop-under-01}
\underline{\gamma} = \min &\left\{ \frac{1}{4 \mu}, \frac{1}{C_1}, \frac{1}{6\sqrt{2} K_0}, \frac{1}{6 \sqrt{2} K_1 (C_1)^{\a / 2}}, \frac{1}{6 \sqrt{2} K_2 (\sqrt{2} R_1 + 1)^{\a /(1 - \a)}} \right\},
\end{align}
and for $\a=1$, it is given by
\begin{align}
\label{eq-th-pop-under-1}
\underline{\gamma} = \min \left\{ \frac{1}{4\mu} , \frac{1}{C_1}, \frac{1}{2 \sqrt{2} (L_0 + L_1 C_1) \exp (\sqrt{2} L_1 R_1 + 1)} \right\}.
\end{align}

The constants $K_0, K_1, K_2$ are as given in Proposition~\ref{prop-a}, while $C_1$ and $\bar{C_1}$ are given by
\begin{align*}
C_1 = \sqrt{2} R_1 (K_0 + K_1 \|F(P_{U^*}(u_1))\|^{\a}
+ K_2 (\sqrt{2} R_1)^{1/(1 - \alpha)}) 
+ \|F(P_{U^*}(u_1))\|,
\end{align*}
\begin{align*}
\bar{C_1} =& \sqrt{2} R_1 (L_0+ L_1 \|F(P_{U^*}(u_1))\|) \exp (\sqrt{2} L_1 R_1) +  \|F(P_{U^*}(u_1))\|,
\end{align*}
and $R_1^2 = \|u_1 - P_{U^*}(u_1)\|^2 + \|u_1 - h_0\|^2$. 

(c) The iterates $u_k$ and $h_k$ converge to a solution $\bar u \in U^*$. 
\end{theorem}
Based on results of Theorem~\ref{thm-Popov-a01-asymp}, we provide convergence rates of Popov method iterates to the solution set. Formally, in the next theorem we provide convergence rates for Popov method~\eqref{eq-popov-det} with adaptive stepsizes~\eqref{pop-step-01} for $\a$-symmpetric operators with $\a \in (0,1]$.

\begin{theorem}
\label{thm-Popov-a01-rates}
Let Assumptions~\ref{asum-alpha}, \ref{asum-set}, \ref{asum-sharp}  hold. Consider the Popov method \eqref{eq-popov-det} with srepsizes $\gamma_k$ given by~\eqref{pop-step-01} for $\a \in (0,1)$ and by ~\eqref{pop-step-1} for $\a =1$, and let
\[R_k^2 = \dist^2(u_k, U^*) + \|u_{k} -h_{k-1}\|^2 \quad \text{for all } k \geq 0.\]
Then, the following results for the quantity $R_k^2$:

\textbf{Case (a)} If operator $F(\cdot)$ is $p$-quasi-sharp with $p=2$, then the following inequality holds for  all $k \ge 1$,
\begin{equation}
    \begin{aligned}
    \label{eq-thm-popov-a01-rate-p-2}
     R_{k+1}^2 & \leq (1 - \mu \underline{\gamma}) R_k^2.
    \end{aligned}
\end{equation}  
\textbf{Case (b)} If operator $F(\cdot)$ is $p$-quasi-sharp with $p>2$,
then the following inequality holds for  all $k \ge 1$,
\begin{equation}
    \begin{aligned}
    \label{eq-thm-popov-a01-rate-p>2}
     R^2_{k+1} & \leq \frac{R_1^2}{ (1 + p C 2^{-p/2}   (R_1^2)^{p/2} k)^{2/p}},
    \end{aligned}
\end{equation}
where $ C = \min \{ \frac{1}{2}(2 R_1^2)^{-(p-2)/2} , 2^{-p+2} \mu \underline{\gamma}\}$. The constant $\underline{\gamma}$ in~\eqref{eq-th-pop-under-01} and~\eqref{eq-th-pop-under-1} is the lower-bound on the stepsizes provided in Theorem~\ref{thm-Popov-a01-asymp}. 
\end{theorem}

\section{Adaptive Clipping}
\label{Adaptive-Clipping}
In practice, constants $L_0, L_1$ and corresponding $K_0, K_1, K_2$ are unknown and have to be estimated. The estimation of such constants can be a difficult problem and lead to large values of $L_0, L_1$. Instead of estimating the parameters of the problem one can fine-tune stepsizes. Moreover, in the recent papers on generalized smoothness (\cite{DBLP:conf/iclr/ZhangHSJ20, DBLP:conf/icml/00020LL23}), the theoretical stepsizes were not applied in numerical results but were fine-tuned. To address the issue of fine-tuning, in this section, we provide analysis for
Korpelevich method with slightly different stepsizes of the form:
\begin{equation}
\label{stepsizes-korp-adap}
\gamma_k = \beta_k \min \left\{1, \frac{1}{\|F(h_k)\|}\right\}. 
\end{equation}
The parameter $\beta_k$ can be easily fine-tuned, while normalization in the stepsizes remains independent from $\alpha, L_0, L_1$ and corresponding $K_0, K_1, K_2$. Moreover, in real-life problems, it is also possible to use different schedule techniques~(\cite{defazio2023and}) for $\beta_k$. Next, we show that Korpelevich method with a simple choice of decreasing sequence $\beta_k$ converges to the solution and has a better practical performance. 

\begin{theorem}
\label{thm-Korpelevich-c2-asym} 
Let Assumptions~\ref{asum-alpha}, \ref{asum-set}, \ref{asum-sharp}  hold. Also, let the stepsizes be given by \eqref{stepsizes-korp-adap} and parameter $\{\beta_k\}$ be such that there exists $N> 0$ such that $ \beta_k (K_0 + K_1 + K_2) < 1$, $ ( L_0 + L_1) \beta_k \exp (L_1 \beta_k)   < 1$ and $\beta_k \leq \frac{1}{4 \mu}$ for all $k \geq N$.
Then,  the following statements hold for the iterate sequences $\{u_k\}$ and $\{h_k\}$ generated by the Korpelevich method~\eqref{eq-korpelevich-det}:

(a) The following descent inequality holds for any solution $u^* \in U^*$
and any $k\geq 0$, 
\begin{equation}
\begin{aligned}
\label{eq-korpelevich-c2-thm-asym}
 \|h_{k+1}  - u^*\|^2 &\leq \|h_k - u^*\|^2 - 2 \g_k \mu \dist^p( u_k, U^*) -(1 - \beta_k^2 C_2(\beta_k))\|u_k - h_k\|^2, \cr
\end{aligned}
\end{equation}
where 
\[C_2(\beta_k) = \max \{ (K_0 +K_1 + K_2)^2, ( L_0 + L_1) \exp (L_1 \beta_k)  \}.\]

(b) The stepsizes sequence $\{\gamma_k\}$ is bounded below, i.e.,
\[\gamma_k \ge   \beta_k \; \min \left\{ 1, \frac{1}{C_1} \right\}  \qquad\hbox{for all } k\geq 0,\]
where the constant $C_1$ is given by
\begin{align*}
C_1 =
\bar{D}_N\max  &\left\{(K_0 + K_1 \|F(P_{U^*}(h_{\bar{N}}))\|^{\alpha} \right. + K_2 (\bar{D}_N)^{ \frac{\a}{ (1 - \a)}}),  \cr
 &(\left.L_0+ L_1 \|F(P_{U^*}(h_{\bar{N}}))\|) \exp (L_1 \bar{D}_N) \right\} + \|F(P_{U^*}(h_{\bar{N}}))\|.
\end{align*}
where $(\bar{D}_N)^2 = \max_{k \in [0, N]} \{ \dist^2(h_k, U^*) \}$, and  $\bar{N} = \argmax_{k \in [0, N]} \{ \dist^2(h_k, U^*) \}$

(c) If $\sum_{k=0}^{\infty} \beta_k = \infty$, then the iterates $u_k$ and $h_k$ converge to a solution $\bar{u} \in U^*$.

\end{theorem}
The proof of these results can be found in Appendix~\ref{App-Korp-c2-asym}. 

\begin{corollary}
Consider Korpelevich method~\eqref{eq-korpelevich-det} with adaptive stepsizes given in~\eqref{stepsizes-korp-adap}. Let  $ \beta_k (K_0 + K_1 + K_2) < 1$, $ ( L_0 + L_1) \beta_k \exp (L_1 \beta_k)   < 1$ and $\beta_k \leq \frac{1}{4 \mu}$ for all $k \geq 0$. Then if $p=2$ and $ \beta_k = \beta$, the method has a linear convergence rate. If $p>2$ then the method has convergence rate of $O((\sum_{i=0}^k \beta_i)^{-2/p})$.
\end{corollary}
These results are direct consequence from equation~\eqref{eq-korpelevich-c2-thm-asym} in Theorem~\ref{thm-Korpelevich-c2-asym} and Lemma 6 (Section 2 in \cite{polyak1987introduction}).
\begin{figure*}[hbt!]
\centering
\subfigure[$(\alpha \approx 0.09, p = 2.1)$]{
\includegraphics[width=.3\textwidth]{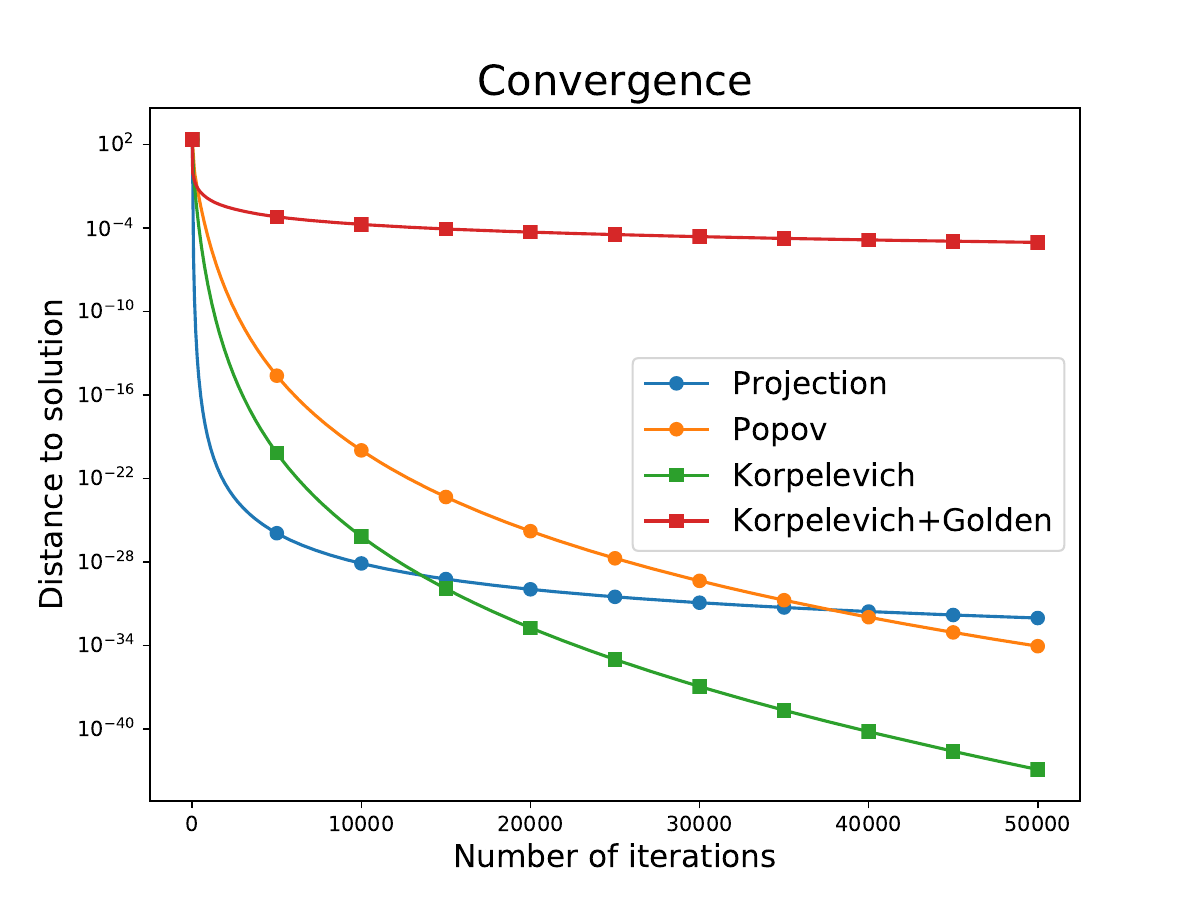}
}
\subfigure[$(\alpha \approx 0.66, p = 4.0)$]{
\includegraphics[width=.3\textwidth]{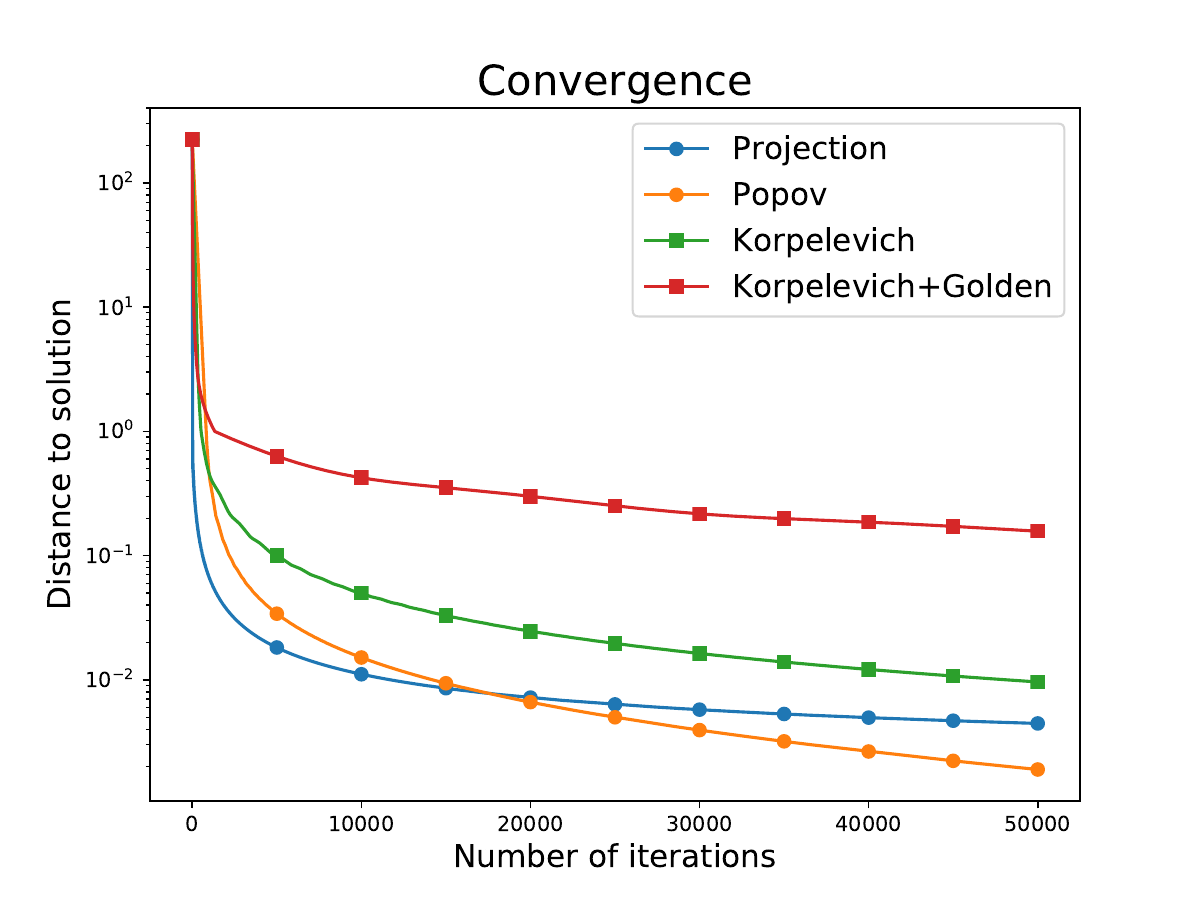}
}
\subfigure[$(\alpha \approx 0.86, p = 8.0)$]{
\includegraphics[width=.3\textwidth]{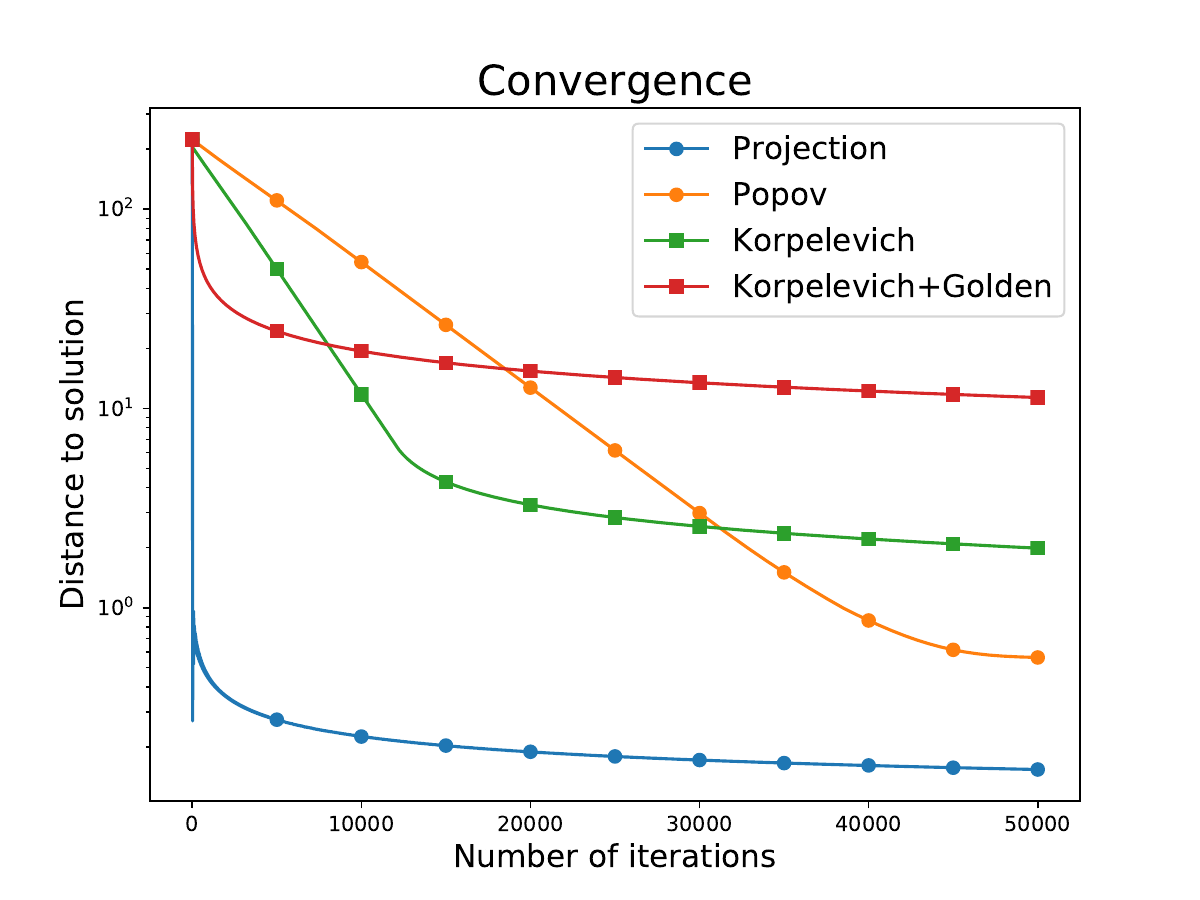}
}
\caption{Comparison of projection, Korpelevich and Popov methods and Korpelevich method with golden ratio stepsizes for different $(\a, p)$. 
}
\label{fig:theory}
\end{figure*}
\begin{figure*}[hbt!]
\centering
\subfigure[$(\alpha \approx 0.09, p = 2.1)$]{
\includegraphics[width=.3\textwidth]{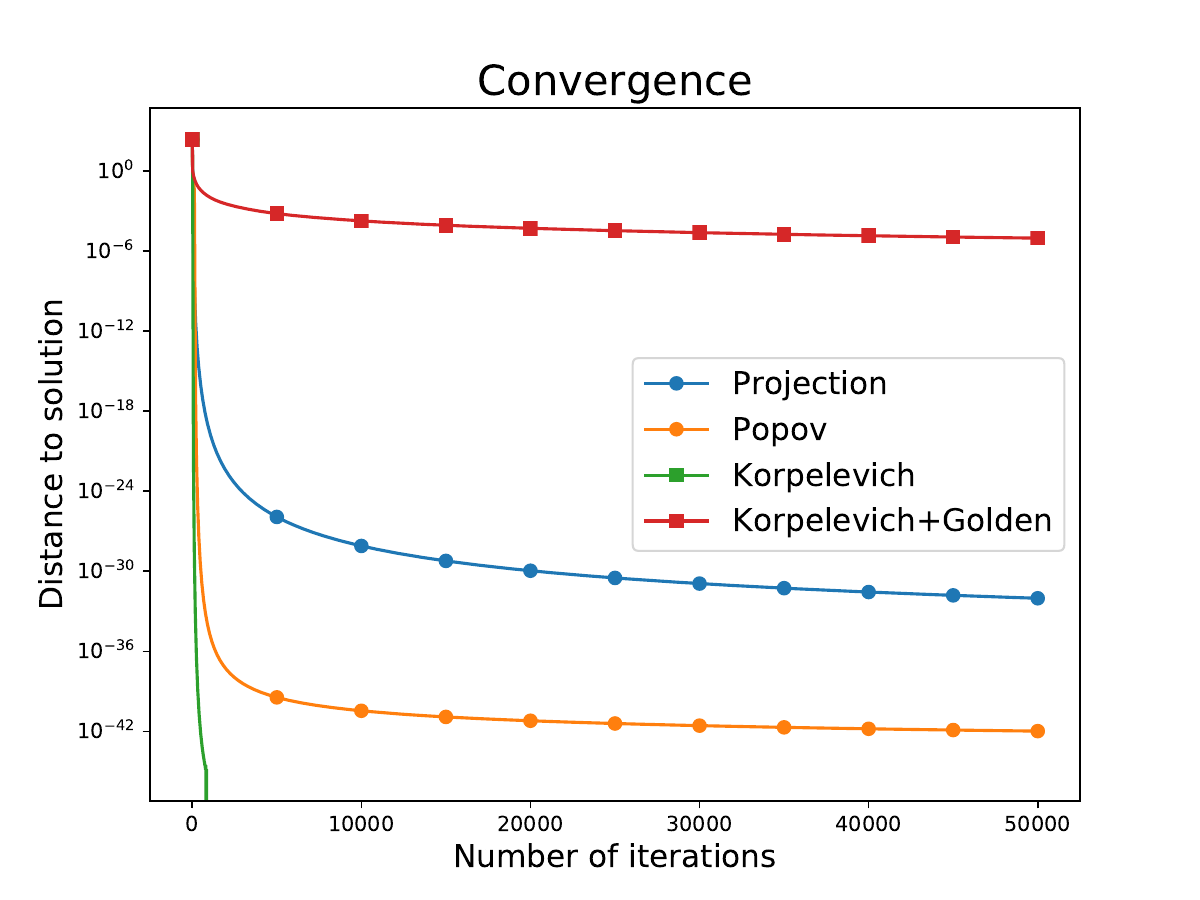}
}
\subfigure[$(\alpha \approx 0.66, p = 4.0)$]{
\includegraphics[width=.3\textwidth]{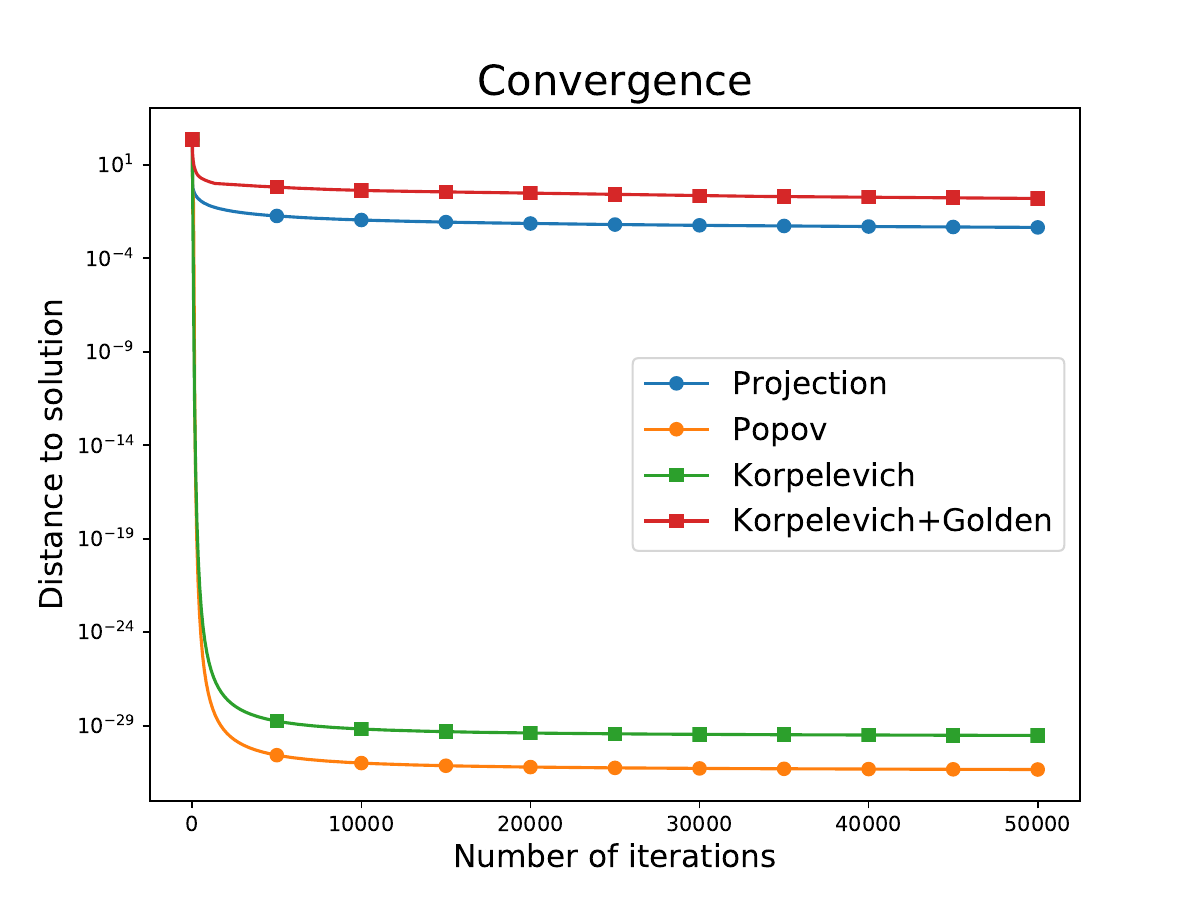}
}
\subfigure[$(\alpha \approx 0.86, p = 8.0)$]{
\includegraphics[width=.3\textwidth]{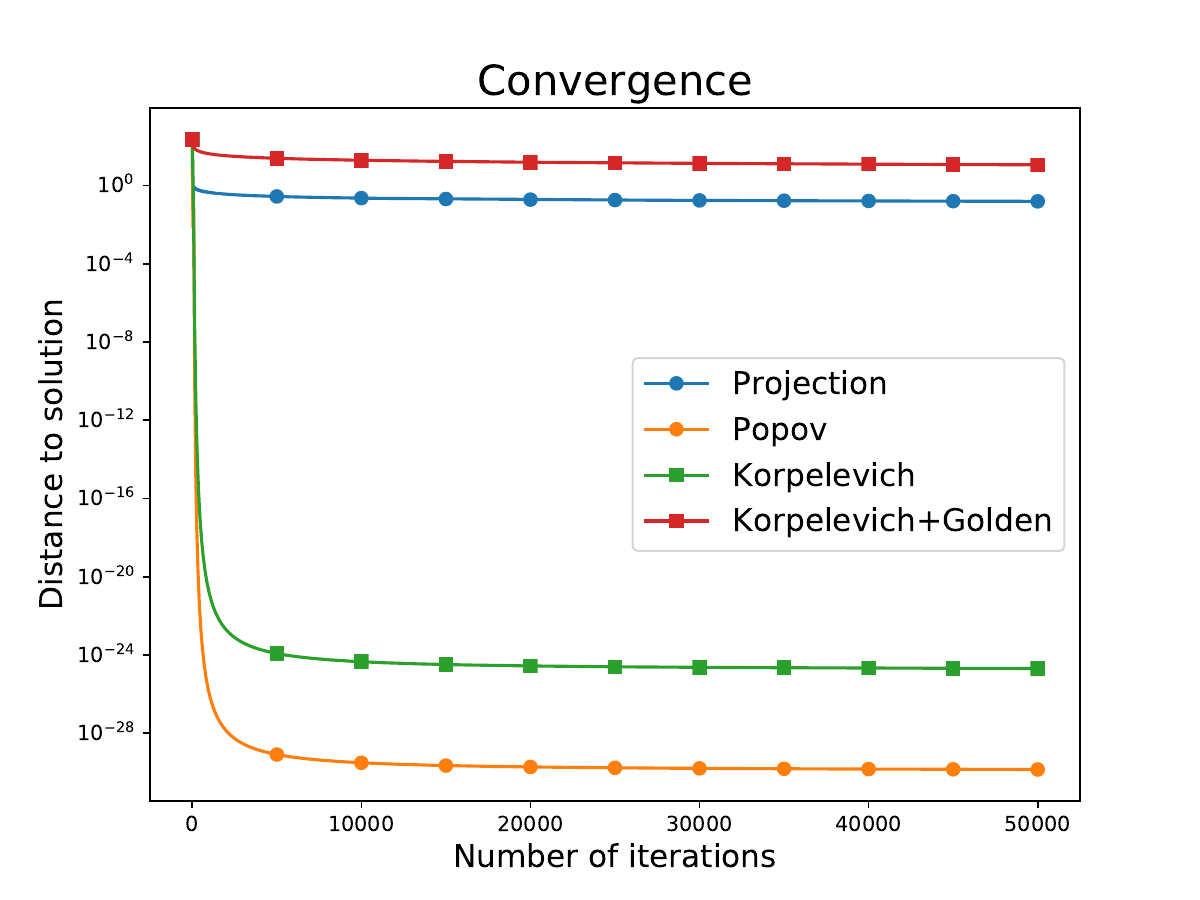}
}
\caption{Comparison of projection, Korpelevich and Popov methods and Korpelevich method with golden ratio stepsizes for different $(\a, p)$. 
}
\label{fig:fair}
\end{figure*}
\section{Numerical Results}
\label{Numerical}
We study the performance of the projection, Korpelevich and Popov methods, for different values of $\a$ and $p$ and under Assumptions~\ref{asum-alpha} and~\ref{asum-sharp}. We consider an unconstrained VI$(\mathbb{R}^2, F)$, where $F(\cdot)$ is an $\a$-symmetric and $p$-quasi sharp operator from Proposition~\ref{example1}. We set parameters of the problem to be $\{(\alpha \approx 0.09 0., p = 2.1), (\alpha \approx 0.66, p = 4.0), (\alpha \approx 0.86, p = 8.0)\}$. 
We also compare our results to the Korpelevich method with golden ratio stepsizes~(\cite{DBLP:journals/tmlr/Bohm23}).

In Figure ~\ref{fig:theory}, we plot the distance to the solution set as a function of the number of iterations. In particular, the stepsizes for the Korpelevich and Popov methods are chosen according to Theorems \ref{thm-Korp-rates} and \ref{thm-Popov-a01-rates}, respectively, while that for the projection method is chosen according to Theorem~\ref{thm-projection-asym} with $\beta_k = \frac{100}{100 + k}$.
We observe that all three methods converge faster than the Korpelevich method with golden-ratio stepsizes for all three choices for $\a,p$. We also notice that when $p$ and $\a$ increase, all methods slow down as predicted by the theory; moreover, projection outperforms other methods. This is due to the fact that stepsizes in Theorems \ref{thm-Korp-rates} and \ref{thm-Popov-a01-rates} depend on the parameters of the problem and decrease significantly when the values  $K_0, K_1, K_2$ increase; on the other hand, the stepsizes in Theorem~\ref{thm-projection-asym} for the projection method are independent of parameters. 

In Figure \ref{fig:fair}, we use fully adaptive stepsizes for both Korpelevich (from 
Theorem~\ref{thm-Korpelevich-c2-asym}) and Popov; for the latter, we use the following schedule: $\gamma_k = \beta_k \min \{1, \frac{1}{\|F(h_k)\|}, \frac{1}{(\|u_k - h_{k-1}\| + 1)^{\a / (1 - \alpha)}}\} $,
with the same $\beta_k = \frac{a}{b + k}$ for all methods, where $a = b = 100$. We choose the parameter $\beta_k = \frac{a}{b + k}$ as it is the simplest diminishing stepsizes one can use in practice. We observe that diminishing stepsizes accelerate both methods and increasing values of the parameters $\a, p$ do not affect performance.


\section{Concluding Remarks}
For a class of structured non-monotone VI problems with $\a$-symmetric and $p$-quasi sharp operators, we have proved the first-known convergence results under the weakest assumption on generalized smoothness. Our convergence rates of $O(k^{-2/p})$ for Korpelevich and Popov methods with adaptive stepsizes improve the existing rate of $O(k^{-1/(p-1)})$ provided in \cite{DBLP:conf/iclr/WeiLZL21} under more restrictive assumption on monotonicity and Lipschitz continuity of the operator. 
Our numerical experiments suggests that for $\a$-symmetric operators, our adaptive stepsizes show more promise than golden ratio stepsizes~(\cite{DBLP:journals/tmlr/Bohm23}).

There are a few potential directions for future research: (i) extending our analysis for weak Minty conditions; (ii) analyzing adaptive methods for solving generalized smooth stochastic VIs; and (iii) extending the results of Theorem~\ref{thm-Korpelevich-c2-asym} for adaptive backtracking search of $\beta_k$ as in \cite{DBLP:journals/siamis/BeckT09}.

\bibliographystyle{abbrvnat}
\bibliography{generalized}

\onecolumn
\
\appendix

\section{Proof of Proposition~\ref{example1}}
\begin{proof}
Notice that the variational inequality VI($\mathbb{R}^2, F$) has a unique solution $u^* = (0,0)$. To see this that the operator is $p$-quasi sharp take an arbitrary $u \in \re^2$. Then, we have:
\begin{align*}
    \langle F(u), u - u^* \rangle & =   \left\langle \begin{bmatrix} \text{sign}(u_1) |u_1|^{p-1} + u_2 \\ \text{sign} (u_2) |u_2|^{p-1} - u_1 \end{bmatrix}, \begin{bmatrix} u_1 \\ u_2 \end{bmatrix} - \begin{bmatrix} 0 \\ 0 \end{bmatrix}  \right\rangle \cr
    &= |u_1|^{p} + |u_2|^{p} \cr
    &\geq 2^{1 - p} |u_1 + u_2|^{p} \cr 
    &\geq 2^{1 - p} \left( \sqrt{u_1^2 + u^2_2} \right)^{p} = 2^{1 - p} \dist^{p}(u, U^*), 
\end{align*}
where the first equality holds due to the Jensen inequality for a convex function $|\cdot|^{p}$, since  $ p \ge 1$. The second inequality holds because due to $\|\cdot\|_1 \geq \|\cdot\|_2$ and monotonicity of $|\cdot|^p$. Thus, for $p\ge 1$, the operator $F(\cdot)$ is $p$-quasi sharp with $\mu=2^{1-p}$.

Next, we show that $F(\cdot)$ is $\a$-symmetric. Consider arbitrary $u, v \in \re^2$. Then, using the triangle inequality we obtain
\begin{align}
    \label{ex1:eq-0}
    \|F(u) - F(v)\| &= \left\| \begin{bmatrix} \text{sign}(u_1) |u_1|^{p - 1} + u_2 \\ \text{sign}(u_2) |u_2|^{p - 1} - u_1 \end{bmatrix} - \begin{bmatrix} \text{sign}(v_1) |v_1|^{p - 1} + v_2 \\ \text{sign}(v_2) |v_2|^{p - 1} - v_1 \end{bmatrix} \right\| \cr
    &\le \|\phi(u)-\phi(v)\|+\|u - v\|,
\end{align}
where $\phi(u) = \begin{bmatrix} \text{sign}(u_1) |u_1|^{p - 1} \\ \text{sign}(u_2) |u_2|^{p - 1}  \end{bmatrix}$. Notice that, for $p > 2$, the operator $\phi(u)$ is differentiable, and its Jacobian is given by
\[\nabla \phi(u) = (p-1) \begin{bmatrix} |u_1|^{p - 2} & 0 \\ 0 & |u_2|^{p - 2}  \end{bmatrix}.\]
Then, by the mean value theorem:

\begin{align}
    \label{ex1:eq-1}
    \|\phi(u) - \phi(v)\| &\leq \|u - v\| \max_{\theta \in (0,1)} \|\nabla \phi(w_{\theta})\|,
\end{align}
where $w_{\theta} = \theta u + (1 - \theta) v$. For the norm of $\nabla \phi(z)$, we have

\begin{align}
\|\nabla \phi(z)\| &= (p-1) \max \{ |z_1|^{p - 2}, |z_2|^{p-2}\} \cr
&= (p-1) \|(|z_1|^{p - 2}, |z_2|^{p-2})'\|_{\infty} \cr
&\leq (p-1) \|(|z_1|^{p - 2}, |z_2|^{p-2})'\|_2 ,
\end{align}
where we used the fact that $\|z\|_{\infty} = \max \{z_1, z_2\}$ and $\|z\|_{\infty} \leq \|z\|_2$ for an arbitrary $z$. Expanding further we obtain:
\begin{align}
\|\nabla \phi(z)\| &\leq (p-1) \sqrt{|z_1|^{2(p - 2)} + |z_2|^{2(p - 2)}} \cr
& =(p-1) \sqrt{\left(|z_1|^{2(p - 1)}\right)^{(p-2)/(p-1)} + \left(|z_2|^{2(p - 1)}\right)^{(p-2)/(p-1)}}.
\end{align}
Since $|\cdot|^{(p-2)/(p-1)}$ is a concave function for $p>1$, it follows that
\begin{align}
\left(|z_1|^{2(p - 1)}\right)^{(p-2)/(p-1)} + \left(|z_2|^{2(p - 1)}\right)^{(p-2)/(p-1)} &\leq 2 \left(\dfrac{|z_1|^{2(p - 1)} + |z_2|^{2(p - 1)}}{2}\right)^{(p-2)/(p-1)} \cr
&=2^{1 /(p-1)} \left(|z_1|^{2(p - 1)} + |z_2|^{2(p - 1)}\right)^{(p-2)/(p-1)}.
\end{align}
Notice that 
\[|z_1|^{2(p - 1)} + |z_2|^{2(p - 1)} = \|(|z_1|^{p - 1}, |z_2|^{p - 1})'\|^2 = \|\phi(z)\|^2.\]
Combining these facts we obtain for all $z\in\re^2$,
\begin{align}
\label{ex1:eq-1-0}
\|\nabla \phi(z) \| &\leq (p-1) 2^{1/2(p-1)} \|\phi(z)\|^{(p-2)/(p-1)}.
\end{align}

Now, we show that 
\begin{equation}
\label{ex1:eq-1-1}
\|\phi (z) \| \leq 2 \|F(z)\| + \sqrt{2} \; 4^{1/(p-2)}.
\end{equation}

By triangle inequality and the definitions of $F(\cdot)$ and $\phi(\cdot)$, it follows that 
\begin{equation}
\label{eq-prop24-tri}
    \|\phi(z)\| \leq \|F(z)\| + \|z\|.
\end{equation}
For $p=2$,
\begin{equation}
\begin{aligned}
    \|z\| &= \sqrt{z_1^2 + z_2^2} \cr
    &= \frac{1}{\sqrt{2}} \sqrt{(z_1 + z_2)^2 + (z_2 - z_1)^2} \cr
    &= \frac{1}{\sqrt{2}} \|F(z)\|. 
\end{aligned}
\end{equation}
Combining this equality with the preceding inequality we conclude, that equation \eqref{ex1:eq-1-1} holds for $p = 2$. Let $p>2$, consider two cases, (1) $\|z\| \leq \sqrt{2} \; 4^{1/(p-2)}$, (2) $\|z\| > \sqrt{2} \; 4^{1/(p-2)}$. For case (1), by inequality~\eqref{eq-prop24-tri}, we see that equation \eqref{ex1:eq-1-1} holds. Now, let $\|z\| > \sqrt{2} \; 4^{1/(p-2)}$. Without loss of generality, assume that $|z_1| \geq |z_2|$, then:
\begin{equation}
\begin{aligned}
    |z_1| &= \|z\|_{\infty} \cr
    &\geq \frac{1}{\sqrt{2}} \|z\|_2 \cr
    &> 4^{1/(p-2)}.
\end{aligned}
\end{equation}
Now, we estimate $\|z\|$, since $|z_1| \geq |z_2|$:
\begin{equation}
    \begin{aligned}
        \|z\|^2 &= |z_1|^2 + |z_2|^2 \cr
        &= |z_1|^2 + |z_2|^2 + |z_1|^{2(p-1)} -  |z_1|^p |z_1|^{p-2} \cr
        &\leq |z_1|^2 + |z_2|^2 + |z_1|^{2(p-1)}  -  |z_1|^p |z_1|^{p-2} + |z_2|^{2(p-1)},
    \end{aligned}
\end{equation}
where in equality we add and subtract $ |z_1|^{2(p-1)}$, and in the inequality we add $|z_2|^{2(p-1)}$. Since $|z_1| > 4^{1/(p-2)}$, then $|z_1|^p \geq 4 |z_1|^2$ and using this fact we obtain:
\begin{equation}
    \begin{aligned}
        \|z\|^2 &\leq |z_1|^2 + |z_2|^2 + |z_1|^{2(p-1)}  -  |z_1|^p |z_1|^{p-2} + |z_2|^{2(p-1)} \cr
        &\leq |z_1|^2 + |z_2|^2 + |z_1|^{2(p-1)}  -  4 |z_1|^2 |z_1|^{p-2} + |z_2|^{2(p-1)} \cr
        &\leq |z_1|^2 + |z_2|^2 + |z_1|^{2(p-1)}  -  2 |z_1|^2 (|z_1|^{p-2} + |z_2|^{p-2}) + |z_2|^{2(p-1)} \cr
        &\leq |z_1|^2 + |z_2|^2 + |z_1|^{2(p-1)}  -  2 |z_1| |z_2| (|z_1|^{p-2} + |z_2|^{p-2}) + |z_2|^{2(p-1)} , \cr
    \end{aligned}
\end{equation}
where in the last two inequalities we used the fact that $|z_1| \geq |z_2|$. Then:
\begin{equation}
    \begin{aligned}
        \|z\|^2 &\leq |z_1|^2 + |z_2|^2 + |z_1|^{2(p-1)}  - 2 |z_1|^{p-1} |z_2| - 2|z_1||z_2|^{p-1} + |z_2|^{2(p-1)} \cr
        &\leq |z_1|^2 + |z_2|^2 + |z_1|^{2(p-1)}  + 2 \text{sign}(z_1) |z_1|^{p-1} |z_2| - 2 \text{sign}(z_2) |z_1||z_2|^{p-1} + |z_2|^{2(p-1)} \cr
        &= \|F(z)\|^2 .
    \end{aligned}
\end{equation}
Therefore, for case (2), $\|z\| \leq \|F(z)\|$. Then it holds that
\begin{align}
\|\phi (z) \| &\leq \|F(z)\| + \|z\| \cr
    &\leq 2 \|F(z)\| + \sqrt{2} \; 4^{1/(p-2)}.
\end{align}

Using equations \eqref{ex1:eq-1-0}, \eqref{ex1:eq-1-1}, and the fact that $\frac{(p-2)}{(p-1)} < 1$ we obtain:
\begin{align}
\|\nabla \phi (w_{\theta})\| &\leq (p-1) 2^{1/2(p-1)} (2 \|F(w_{\theta})\| + \sqrt{2} \; 4^{1/(p-2)})^{(p-2)/(p-1)} \cr
&\leq 2 (p-1) 2^{1/2(p-1)}  \|F(w_{\theta})\|^{(p-2)/(p-1)}  + (p-1) 2^{1/2} 4^{1/(p-1)}.
\end{align}
  
Then by taking maximum over $\theta \in (0,1)$ on both sides we obtain

\begin{align}
\label{ex1:eq-2}
\max_{\theta \in (0,1)} \|\nabla \phi (w_{\theta})\| &\leq (p-1) 2^{1/2} 4^{1/(p-1)} + 2(p-1) 2^{1/2(p-1)} \max_{\theta \in (0,1)}  \|F(w_{\theta})\|^{(p-2)/(p-1)}.
\end{align}
 Combining equations~\eqref{ex1:eq-0}, \eqref{ex1:eq-1}, and \eqref{ex1:eq-2} we obtain

\begin{align}
    \label{ex1:eq-3}
    \|F(u) - F(v)\| &\leq (1+ (p-1) 2^{1/2} 4^{1/(p-1)} + 2(p-1) 2^{1/2(p-1)} \max_{\theta \in (0,1)}  \|F(w_{\theta})\|^{(p-2)/(p-1)}) \|u - v\|.
\end{align}
Thus, the operator is $\alpha$-symmetric with $L_0 = 1+ (p-1) 2^{1/2} 4^{1/(p-1)}, L_1 = 2(p-1) 2^{1/2(p-1)} $ and $\alpha= \frac{p-2}{p-1}$.
\end{proof}

\section{Technical Lemmas}
\label{App-Tech-Lemmas}
In our analysis, we use the properties of the projection operator $P_U(\cdot)$ given in the following lemma.
\begin{lemma}\label{lem-proj} (Theorem 1.5.5 and Lemma 12.1.13 in~\cite{facchinei2003finite}) Given a nonempty convex closed set $U\subset\mathbb{R}^\bd,$ the projection operator $P_U(\cdot)$ has the following properties:
\begin{equation}
    \label{eq-proj1}
    \langle v - P_{U}(v), u - P_{U}(v) \rangle \leq 0  \quad \hbox{for all } u \in U, v \in \mathbb{R}^\bd,
\end{equation}
\begin{equation}
    \label{eq-proj3} 
    \|u - P_{U}(v)\|^2 \leq \|u - v\|^2 - \|v - P_{U}(v)\|^2 \quad \hbox{for all } u \in U, v \in \mathbb{R}^\bd,
\end{equation}
\begin{equation}
    \label{eq-proj5}
    \|P_U(u)-P_U(v)\|\le \|u-v\| \quad \hbox{for all } u , v \in \mathbb{R}^\bd.
  \end{equation}  
\end{lemma}


In the forthcoming analysis, we use Lemma~11 \cite{polyak1987introduction}, which is stated below.
\begin{lemma} \label{lemma-polyak11} [Chapter 2,  Lemma~11 \cite{polyak1987introduction}]
Let $\{v_k\}, \{z_k\}, \{a_k\},$ and  $\{b_k\}$ be nonnegative random scalar sequences such that almost surely for all $k\ge0$,
\begin{equation}
\begin{aligned}
\label{eq-polyak-0}
\mathbb{E}[v_{k+1}\mid {\cal F}_k] \leq &(1 + a_k)v_k -z_k + b_k,
\end{aligned}
\end{equation}
where 
${\cal F}_k = \{v_0, \ldots, v_k, z_0, \ldots, z_k, a_0, \ldots, a_k,b_0, \ldots,b_k\}$, and
\emph{a.s.} $\sum_{k=0}^{\infty} a_k < \infty$, $\sum_{k=0}^{\infty} b_k < \infty$. Then, almost surely, $\lim_{k\to\infty} v_k =v $  for some nonnegative random variable $v$ and $\sum_{k=0}^{\infty} z_k < \infty$.
\end{lemma}

As a direct consequence of Lemma~\ref{lemma-polyak11}, when
the sequences $\{v_k\}, \{z_k\}, \{a_k\}, \{b_k\}$ are deterministic, we obtain the following result.
\begin{lemma} \label{lemma-polyak11-det} 
Let $\{\bar v_k\}, \{\bar z_k\}, \{\bar a_k\}, \{\bar b_k\}$ be nonnegative scalar sequences such that for all $k\ge0$,
\begin{equation}
\begin{aligned}
\label{eq-polyak-1}
\bar v_{k+1}\leq &(1 + \bar a_k)\bar v_k -\bar z_k + \bar b_k,
\end{aligned}
\end{equation}
where $\sum_{k=0}^{\infty} \bar a_k < \infty$ and $\sum_{k=0}^{\infty} \bar b_k < \infty$. Then, 
$\lim_{k\to\infty} \bar v_k = \bar v$ 
for some scalar $\bar v\ge0$ and $\sum_{k=0}^{\infty} \bar z_k < \infty$.
\end{lemma}

    
\begin{lemma}
\label{inequality-hoed-2} 
Let $a_1, a_2$ be nonnegative scalar and $p>0$. Then the following inequality holds:
\[\left( a_1 + a_2\right)^p \leq 2^{p-1} (a_1^p + a_2^p).  \]
\end{lemma}
\begin{proof}
Let $a = (a_1, a_2), b = (1,1)$, then by H\"older's inequality:
\begin{align*}
a_1 + a_2 &= |\langle a, b \rangle | \cr
&\leq \|a\|_{p} \|b\|_{p/(p-1)} \cr
&\leq (a_1^p + a_2^p)^{1/p} (1 + 1)^{(p-1)/p}.
\end{align*}
Raising the inequality in the power $p$ we get the desired relation.
\end{proof}

\begin{lemma}
\label{lemma-polyak-6} [Chapter~2, Lemma~6 \cite{polyak1987introduction}]
Let $\{x_k\}$, $\{a_k\}$ be nonnegative scalar sequences, and $q>0$ such that for all $k\ge0$,
\begin{equation}
\begin{aligned}
\label{eq-polyak6-1}
x_{k+1} \leq x_k  - a_k x_k^{1 + q}.
\end{aligned}
\end{equation}
Then
\begin{equation}
\begin{aligned}
\label{eq-polyak6-2}
x_{k} \leq \dfrac{x_0}{(1 + q x_0^q \sum_{i=0}^{k-1} a_i)^{1/q}}.
\end{aligned}
\end{equation}
In particular, when $a_k = a$, then

\begin{equation}
\begin{aligned}
\label{eq-polyak6-3}
x_{k} \leq \dfrac{x_0}{(1 + a q x_0^q k )^{1/q}}.
\end{aligned}
\end{equation}

\end{lemma}

\section{Projection Method Analysis}
Proof of Theorem~\ref{thm-projection-asym}.
\begin{proof}
Let $k\ge 0$ be arbitrary but fixed. From the definition of $u_{k+1}$ in~\eqref{eq-proj-det}, we have $\|u_{k+1} - y\|^2 = \|P_{U}(u_k - \gamma_k F(u_k)) - y\|^2$ for all $y \in U$. Using the non-expansiveness property of projection operator~\eqref{eq-proj5} we obtain for all $y \in U$ and $k\ge0$,
\begin{align}\label{eq-grad-1}
\|u_{k+1}-y\|^2 
&\le \|u_k - \g_k F(u_k) - y\|^2\cr
&= \|u_k -y\|^2 - 2\g_k \langle F(u_k), u_{k}-y\rangle + \g_k^2 \|F(u_k)\|^2.
\end{align}
By the definition of the stepsizes~\eqref{proj-stepsizes}, $\gamma_k = \beta_k \min\{1, \frac{1}{\|F(u_k)\|}\}$. Thus, we obtain:

\begin{align}\label{eq-grad-2}
\|u_{k+1}-y\|^2 
&\le \|u_k - \g_k F(u_k) - y\|^2\cr
&= \|u_k -y\|^2 - 2\g_k \langle F(u_k), u_{k}-y\rangle + \beta_k^2 \min\{1, \frac{1}{\|F(u_k)\|^2}\} \|F(u_k)\|^2 \cr
&\leq \|u_k -y\|^2 - 2\g_k \langle F(u_k), u_{k}-y\rangle + \beta_k^2 .
\end{align}

Plugging in $y = u^* \in U^*$, where $u^*$ is an arbitrary solution, and using $p$-quasi sharpness we get:
\begin{align}\label{eq-grad-3}
\|u_{k+1}-u^*\|^2 &\leq \|u_k -u^*\|^2 - 2\g_k \langle F(u_k), u_{k}-u^*\rangle + \beta_k^2 \cr
&\leq \|u_k -u^*\|^2 - 2\g_k \dist^p(u_k, U^*) + \beta_k^2 .
\end{align}

The equation~\eqref{eq-grad-3} satisfies the condition of Lemma~\ref{lemma-polyak11-det} with 
\begin{align}\bar v_{k} = \|u_{k}  - u^*\|^2, \quad  \bar a_k = 0, \cr
\bar z_k =  2  \mu \gamma_k \, \dist^p(u_{k}, U^*), \quad \bar b_k = \beta_k^2 .
\end{align}

By Lemma~\ref{lemma-polyak11}, it follows that the sequence $\{\bar v_k\}$ converges to a non-negative scalar for any $u^*\in U^*$. Therefore, for all solutions $u^* \in U^*$ the sequence $\{\|u_k - u^*\| \}$ is bounded by $B(u^*)$.

\textbf{(b)} Now, we want to estimate $\|F(u_k)\|$, since this term is present in the denominator of the stepsize. We add and subtract $F(u^*_0)$, where $u^*_0 = P_{U^*}(u_0)$ is a projection of $u_0$ onto the solution set $U^*$, and get $\|F(u_{k})\| = \|F(u_{k}) - F(u^*_0) + F(u^*_0)\| \leq \|F(u_{k}) - F(u^*_0)\| + \|F(u^*_0)\|$.
We can estimate the first term using the $\alpha$-symmetric assumption on the operator.

\textbf{Case(I) $\alpha \in (0,1)$}
\begin{equation}
\begin{aligned}
\label{eq-grad-10}
\|F(u_k) - F(u^*)\| \leq \|u_k - u_0^*\| (K_0 + K_1 \|F(u^*_0)\|^{\alpha} + K_2 \|u_k - u^*_0 \|^{\a / (1 - \a)}).
\end{aligned}
\end{equation}
Earlier, in part \textbf{(a)} we proved that for the solution $u^*$ the following bound hold for any $k \geq 0$ 
\[\|u_{k} - u^*_0\|   \leq B(u^*_0) .\]

Using this fact and equation~\eqref{eq-grad-10} for we obtain that for all $k \ge 0$:

\begin{align}\label{eq-grad-11}
\|F(u_k) \| \leq B(u^*_0) (K_0 + K_1 \|F(u^*_0)\|^{\alpha} + K_2 B(u^*_0)^{\a / (1 - \a)}) + \|F(u^*_0)\|.
\end{align}

We showed that the sequence \{$\|F(u_k)\|$\} is upper bounded by some constant $C_1$. Where $C_1 = B(u^*_0) (K_0 + K_1 \|F(u^*_0)\|^{\alpha} + K_2 B(u^*_0)^{\a / (1 - \a)}) + \|F(u^*_0)\|$. Using these facts, we conclude that for all $k \geq 0$:
\begin{align}
\label{eq-grad-12}
\g_k &= \beta_k \min \{ 1 , \frac{1}{\|F(u_k)\|} \} \cr
&\geq \min \{ 1, \frac{1}{C_1} \} .
\end{align}

\textbf{Case(II) $\alpha = 1$}

For $\a=1$ by Proposition~\ref{prop-a} we have
\begin{align}\label{eq-grad-13}
\|F(u_{k}) -  F(u^*_0)\| &\leq \|u_k - u^*_0\| (L_0+ L_1 \|F(u^*_0)\|) \exp (L_1 \|u_{k} - u^*_0\|) .
\end{align}
We can get the following bound for $\|F(u_{k})\|$, for all $k \ge 0$:
\begin{align}\label{eq-grad-14}
\|F(u_{k})\| &\leq \|F(u_{k}) - F(u^*_0)\| + \|F(u^*_0)\| \cr
&\leq \|u_k - u_0^*\| (L_0+ L_1 \|F(u^*_0)\|) \exp (L_1 \|u_{k} - u^*_0\|)+  \|F(u^*_0)\| . 
\end{align}
Earlier, in part (a) we proved that for the solution $u^*$ the following bound holds for any $k \geq 0$ 
\[\|u_{k} - u^*\| \leq  B(u^*_0) .\]

Using these facts and equation~\eqref{eq-grad-14} for all $k \ge 0$:

\begin{align}\label{eq-grad-15}
\|F(u_{k})\| &\leq  \|u_k - u^*_0\| (L_0+ L_1 \|F(u^*_0)\|) \exp (L_1 \|u_{k} - u^*_0\|) + \|F(u^*)\| \cr
&\leq \|u_k - u^*\| (L_0+ L_1 \|F(u^*_0)\|) \exp (L_1 B(u^*_0) ) +  \|F(u^*_0)\| .
\end{align}
We showed that the sequence \{$\|F(u_k)\|$\} is upper bounded by some constant $\overline{C}_1$, where $\overline{C}_1 = B(u^*_0)  (L_0+ L_1 \|F(u^*_0)\|) \exp (L_1 B(u^*_0) ) + \|F(u^*_0)\| $. Then, we conclude that 
\begin{equation}
\begin{aligned}
\label{eq-grad-16}
\g_k &= \beta_k \min \{ 1, \frac{1}{\|F(u_k)\|}\} \cr
&\geq \beta_k \min \{ 1 , \frac{1}{\overline{C}_1} \} 
\end{aligned}
\end{equation}

For both cases $\a \in (0,1)$ and $\alpha=1$ in equations~\eqref{eq-grad-12}, ~\eqref{eq-grad-16} we showed that stepsizes sequence $\{\gamma_k\}$ is lower bounded by $\beta_k \min \{1, \frac{1}{C_1}, \frac{1}{\overline{C}_1}\}$.

\textbf{(c)}
Now we show that $\{u_k\}$ converges to a solution of $VI(U, F)$.
Using the provided bound on the stepsize sequence $\gamma_k$ from part (b) and equation~\eqref{eq-grad-3} we obtain

\begin{align}\label{eq-grad-18}
\|u_{k+1}-u^*\|^2
&\leq \|u_k -u^*\|^2 - 2\g_k \dist^p(u_k, U^*) + \beta_k^2  \cr
&\leq \|u_k -u^*\|^2 - 2 \beta_k \min \{1, \frac{1}{C_1}, \frac{1}{\overline{C}_1}\} \, \dist^p(u_k, U^*) + \beta_k^2
\end{align}

The equation~\eqref{eq-grad-18} satisfies the condition of Lemma~\ref{lemma-polyak11-det} with 
\begin{align}\bar v_{k} = \|u_{k}  - u^*\|^2, \quad  \bar a_k = 0, \cr
\bar z_k =  2  \mu \beta_k \min \{1, \frac{1}{C_1}, \frac{1}{\overline{C}_1}\} \, \dist^p(u_{k}, U^*), \quad \bar b_k = \beta_k^2 .
\end{align}

Then
\[
\sum_{k=0}^{\infty}  \beta_k \dist^p(u_k, U^*) < \infty.\]
Since, $\sum_{k=0}^{\infty} \beta_k = \infty$ then it follows that 
\begin{equation}\label{eq-proj-an0-det-a1}
\liminf_{k\to\infty}\dist^p(u_k, U^*) =0 .
\end{equation}

Since $\bar v_k$ converges for any given $u^* \in U^*$  we can conclude that $\|u_k - u^*\|$ converges for all $u^* \in U^*$. Therefore, the sequence $\{u_k\}$ is bounded and has accumulation points. Let $\{k_{i}\}$ be an index sequence, such that 

\begin{equation}\label{eq-proj-an0-det-a2}
\lim_{i \rightarrow \infty}  \dist^p(u_{k_i}, U^*) = \liminf_{k\to\infty}\dist^p(u_k, U^*) =0.
\end{equation}

We assume, that the sequence $\{u_{k_i}\}$ is convergent with a limit point $\bar{u}$, otherwise, we choose a convergent subsequence. Therefore,
\begin{equation}\label{eq-proj-an0-det-a3}
\lim_{i \rightarrow \infty}  \|u_{k_i} - \bar u \| = 0.
\end{equation}
Then  by~\eqref{eq-proj-an0-det-a2},  $\dist (\bar{u}, U^*) = 0$, thus $\bar u \in U^*$ since $U^*$ is closed. And since the sequence $\{\|u - u^*\|\}$ converges for all $u^* \in U^*$ and by~\eqref{eq-proj-an0-det-a3} we have
\begin{equation}\label{eq-proj-an0-det-a4}
\lim_{k \rightarrow \infty}  \|u_{k} - \bar u \| = 0.
\end{equation}

\end{proof}

\section{Korpelevich Method Analysis}
\label{App-Korp}

Proof of Lemma~\ref{Lemma1-Korp}.
\begin{proof}
Let $k\ge 0$ be arbitrary but fixed. By the definition of $u_{k+1}$ in~\eqref{eq-korpelevich-det}, we have 
$\|h_{k+1} - y\|^2 = \|P_{U} (h_{k} - \g_k  F(u_k)) - y\|$ for any $y \in U$. Using the projection inequality
~\eqref{eq-proj3} of Lemma~\ref{lem-proj}
we obtain for any $y\in U$,
\begin{equation}
\begin{aligned}
\label{eq-lemma-korp-2}
 \|h_{k+1}  - y\|^2 & \leq \|h_{k} - \g_k F(u_k) - y\|^2 - \|h_{k+1} -h_{k} + \g_k F(u_k)\|^2\\
& \leq \|h_k - y\|^2 - \| h_{k+1} - h_k\|^2 + 2 \g_k \langle F(u_k), y -  h_{k+1}\rangle.
\end{aligned}
\end{equation}
Next, we consider the term $\|h_{k+1} - h_k\|^2$, where we add and subtract $u_k$, and thus get
\begin{equation}
\begin{aligned}
\label{eq-lemma-korp-3}
 \|h_{k+1}  - h_k\|^2 & = \|h_{k+1} - u_k\| + \|h_{k} - u_k\|^2 - 2 \langle h_{k+1} - u_k, h_k - u_k \rangle.
\end{aligned}
\end{equation}

Adding and subtracting $ 2 \g_k \langle F(h_k), u_k-h_{k+1}\rangle$, and combining~\eqref{eq-lemma-korp-2} and~\eqref{eq-lemma-korp-3} we obtain
\begin{equation}
\begin{aligned}
\label{eq-lemma-korp-4}
 \|h_{k+1}  - y\|^2 & \leq \|h_k - y\|^2 - \| h_{k+1} - u_k \|^2 - \|h_k - u_k\|^2 +2 \langle h_{k+1} - u_k, h_k - u_k \rangle \cr
 &+ 2 \g_k \langle F(u_k), y - u_k + u_k-  h_{k+1}\rangle  + 2 \g_k \langle F(h_k), u_k-  h_{k+1}\rangle - 2 \g_k \langle F(h_k), u_k-  h_{k+1}\rangle\\
 &\leq \|h_k - y\|^2 - \| h_{k+1} - u_k \|^2 - \|h_k - u_k\|^2 +2 \langle h_{k+1} - u_k, h_k - \g_k F(h_k) - u_k \rangle \cr
 &+ 2 \g_k \langle F(u_k), y - u_k \rangle + 2\g_k   \langle F(h_k) - F(u_k), h_{k+1} - u_k \rangle . \\
\end{aligned}
\end{equation}
By the projection inequality~\eqref{eq-proj1} where $u = h_{k+1}$, and $v= h_k - \g_k F(h_k)$ we obtain that \[2 \langle h_{k+1} - u_k, h_k - \g_k F(h_k) - u_k \rangle \leq 0 .\]
Now using Cauchy-Schwartz inequality and well know relation $2ab \leq a^2 + b^2$ for $a,b \in \mathbb{R}$, we obtain:
\begin{align}
2\g_k   \langle F(h_k) - F(u_k), h_{k+1} - u_k \rangle &\leq 2 \g_k \|F(h_k) - F(u_k) \| \|h_{k+1} - u_k\| \cr
&\leq \g_k^2 \|F(h_k) - F(u_k)\|^2 + \|h_{k+1} - u_k\|^2 .
\end{align}

Combining the preceding two estimates with~\eqref{eq-lemma-korp-4}, we obtain the desired result:
\begin{equation}
\begin{aligned}
\label{eq-lemma-korp-5}
 \|h_{k+1}  - y\|^2 &\leq \|h_k - y\|^2  - \|h_k - u_k\|^2 - 2 \g_k \langle F(u_k), u_k - y \rangle + \g_k^2 \|F(h_k) - F(u_k)\|^2 .
\end{aligned}
\end{equation}
\end{proof}

\textbf{Proof of Theorem~\ref{thm-Korpelevich-asym}}
\label{App-Korp-asym}
\begin{proof}
\textbf{(a)}
By Lemma~\ref{Lemma1-Korp}, the  following inequality holds for any $y \in U$ and all $k\ge0$:
\begin{equation}
\begin{aligned}
\label{eq-korpelevich-0}
 \|h_{k+1}  - y\|^2 &\leq \|h_k - y\|^2  - \|h_k - u_k\|^2 - 2 \g_k \langle F(u_k), u_k - y \rangle + \g_k^2 \|F(h_k) - F(u_k)\|^2.
\end{aligned}
\end{equation}
We want to estimate the last term on the LHS of the inequality using the fact that operator $F(\cdot)$ is an $\a$-symmetric operator.

\textbf{Case (I) $\a \in (0,1)$}. 

Based on the alternative characterization of $\a$-symmetric operators from Proposition~\ref{prop-a}(a) (as given in ~\eqref{alpha-property}), when $\a \in (0,1)$, the next inequality holds for any $k\ge1$:
\begin{equation}
\begin{aligned}
\label{eq-korpelevich-1-0}
\|F(h_k) - F(u_k)\| \leq \|h_k - u_k\| (K_0 + K_1 \|F(h_k)\|^{\alpha} + K_2 \|h_k - u_k \|^{\a / (1 - \a)}).
\end{aligned}
\end{equation}
By the projection property~\eqref{eq-proj3} and the stepsize choice~\eqref{korp-step-01} :
\[\| h_{k} - u_k\| \leq \g_k \| F(h_{k})\| \leq 1. \]
Then by the well know relation $(\sum_{i=1}^m a_i)^2\le m\sum_{i=1}^m a_i^2$ we obtain:

\begin{equation}
\begin{aligned}
\label{eq-korpelevich-1-1}
\|F(h_k) - F(u_k)\|^2 &\leq \|h_k - u_k\|^2 (K_0 + K_1 \|F(h_k)\|^{\alpha} + K_2 )^2 \cr
&\leq 3 K_0^2 \|h_k - u_k\|^2 + 3 K_1^2 \|F(h_k)\|^{2 \alpha} \|h_k - u_k\|^2 + 3 K_2^2  \|h_k - u_k\|^2 . 
\end{aligned}
\end{equation}
Based on the step-size choice: $\g_k \leq \frac{1}{3 \sqrt{2} K_0 } $, $\g_k \leq \frac{1}{3 \sqrt{2} K_1\|F(h_{k})\|^{\alpha}}$, $\g_k \leq \frac{1}{3 \sqrt{2} K_2 } $ and  then:
\begin{equation}
\begin{aligned}
\label{eq-korpelevich-1-2}
\g_k^2 \|F(h_{k}) -  F(u_k)\|^2 &\leq \frac{1}{6}\|h_k - u_k\|^2 + \frac{1}{6}\|h_k - u_k\|^2 + \frac{1}{6}\|h_k - u_k\|^2 \cr
&= \frac{1}{2} \|h_k - u_k\|^2 .
\end{aligned}
\end{equation}

\textbf{Case (II) $\a=1$}.

Based on the alternative characterization of $\a$-symmetric operators from Proposition~\ref{prop-a}(b) (as given in ~\eqref{alpha-property-1}), when $\a = 1$, the next inequality holds for any $k\ge0$:
\begin{equation}
\begin{aligned}
\label{eq-korpelevich-2-0}
\|F(h_{k}) -  F(u_k)\| &\leq \|h_k - u_k\| (L_0+ L_1 \|F(h_k)\|) \exp (L_1 \|h_{k} - u_k\|).
\end{aligned}
\end{equation}
By the projection property, triangle inequality and step size choice~\eqref{korp-step-1}:
\[\| h_{k} - u_k\| \leq  \g_k \| F(h_k)\| \leq \frac{1}{ 2 \sqrt{2} e L_1} < \frac{1}{L_1}. \]
Using the well know relation $(\sum_{i=1}^m a_i)^2\le m\sum_{i=1}^m a_i^2$ we obtain:
\begin{equation}
\begin{aligned}
\label{eq-korpelevich-2-1}
\|F(h_{k}) -  F(u_k)\|^2 &\leq \|h_k - u_k\|^2 (L_0+ L_1 \|F(h_{k})\|)^2 e^2 \cr
&\leq 2 e^2 L_0^2 \|h_k - u_k\|^2 + 2 L_1^2 e^2  \|F(h_{k})\|^2 \|h_k - u_k\|^2  .
\end{aligned}
\end{equation}
Using the step-size choice: $\g_k \leq \frac{1}{2\sqrt{2} e L_0}$,  $\g_k \leq \frac{1}{2 \sqrt{2} L_1 e \|F(h_{k})\|}$, it follows that:
\begin{equation}
\begin{aligned}
\label{eq-korpelevich-2-2}
\g_k^2 \|F(h_{k}) -  F(u_k)\|^2 &\leq \frac{1}{2}\|h_k - u_k\|^2.
\end{aligned}
\end{equation}
For both cases $\g_k^2 \|F(h_{k}) -  F(u_k)\|^2 \leq \frac{1}{2}\|h_k - u_k\|^2$, then combining this fact with~\eqref{eq-korpelevich-0} we obtain that for any $k \ge 0$:

\begin{equation}
\begin{aligned}
\label{eq-korpelevich-3}
 \|h_{k+1}  - y\|^2 &\leq \|h_k - y\|^2  -\frac{1}{2}\|u_k - h_k\|^2 - 2 \g_k \langle F(u_k), u_k - y \rangle .
\end{aligned}
\end{equation}


Now we plug $y=u^*$ into equation~\eqref{eq-korpelevich-3}, where $u^* \in U^*$ is an arbitrary solution. Thus, by using $p$-quasi sharpness of the operator, we obtain the following recursive inequality:
\begin{equation}
\begin{aligned}
\label{eq-korpelevich-4}
 \|h_{k+1}  - u^*\|^2 &\leq \|h_k - u^*\|^2  -\frac{1}{2}\|u_k - h_k\|^2 - 2 \g_k \mu \dist^p( u_k, U^*).
\end{aligned}
\end{equation}
From this inequality, we conclude that for any $u^* \in U^*$, the sequence $\{\|h_k - u^*\|\}$ is bounded.

\textbf{(b)} Now, we want to estimate $\|F(h_k)\|$, since this term is present in the denominator of the stepsize. Let $h^*_0 = P_{U^*}(h_0)$ be a projection of $h_0$ onto the solution set $U^*$. We add and subtract $F(h^*_0)$ , and get $\|F(h_{k})\| = \|F(h_{k}) - F(h^*_0) + F(h^*_0)\| \leq \|F(h_{k}) - F(h^*_0)\| + \|F(h^*_0)\|$. We say that $\|F(h^*_0)\| = D_2$ where $D_2 > 0$ is some constant. We can estimate the first term using the $\alpha$-symmetric assumption on the operator class.

\textbf{Case(I) $\alpha \in (0,1)$}

Based on the alternative characterization of $\a$-symmetric operators from Proposition~\ref{prop-a}(a) (as given in ~\eqref{alpha-property}), when $\a \in (0,1)$, the next inequality holds for any $k\ge0$:
\begin{equation}
\begin{aligned}
\label{eq-korpelevich-5}
\|F(h_k) - F(h_0^*)\| \leq \|h_k - h_0^*\| (K_0 + K_1 \|F(h_0^*)\|^{\alpha} + K_2 \|h_k - h_0^* \|^{\a / (1 - \a)}).
\end{aligned}
\end{equation}
Earlier, in equation \eqref{eq-korpelevich-4} we proved that for arbitrary solution $u^*$ the following bound hold for any $k \geq 0$ 
\[\|h_{k+1} - u^*\|^2   \leq \|h_0 - u^*\|^2 .\]

Let $R_0 = \|h_0 - h^*_0\|^2$, then  
\[ \|h_k - h^*_0\| \leq \sqrt{R_0} .\]
Using this fact and equation~\eqref{eq-korpelevich-5} for we obtain that for all $k \ge 0$:

\begin{align}\label{eq-korpelevich-6}
\|F(h_k) \| \leq \sqrt{R_0} (K_0 + K_1 D_2^{\alpha} + K_2 \sqrt{R_0}^{\a / (1 - \a)}) + D_2.
\end{align}

We showed that the sequence \{$\|F(h_k)\|$\} is upper bounded by some constant $C_1$. Where $C_1 = \sqrt{R_0} (K_0 + K_1 D_2^{\alpha} + K_2 \sqrt{R_0}^{\a / (1 - \a)}) + D_2$. Since function $x^a$ is monotone for any $x>0$, and $a > 0$ then  $\|F(h_{k})\|^{\alpha} \leq C_1^{\alpha}$. Using these facts, we conclude that 
\begin{align}
\label{eq-korpelevich-7}
\g_k &= \min \{ \frac{1}{4\mu} ,\frac{1}{3 \sqrt{2} K_0}, \frac{1}{\|F(h_{k})\|}, \frac{1}{ 3 \sqrt{2} K_1 \|F(h_{k})\|^{\a}}, \frac{1}{ 3 \sqrt{2} K_2 } \} \cr
&\geq \min \{ \frac{1}{4\mu} ,\frac{1}{3 \sqrt{2} K_0}, \frac{1}{C_1}, \frac{1}{ 3 \sqrt{2} K_1 C_1^{\a}}, \frac{1}{ 3 \sqrt{2} K_2 }\} 
\end{align}

\textbf{Case(II) $\alpha = 1$}
Based on the alternative characterization of $\a$-symmetric operators from Proposition~\ref{prop-a}(b) (as given in ~\eqref{alpha-property-1}), when $\a = 1$, the next inequality holds for any $k\ge0$:
\begin{align}\label{eq-korpelevich-8}
\|F(h_{k}) -  F(h_0^*)\| &\leq \|h_k - h_0^*\| (L_0+ L_1 \|F(h_0^*)\|) \exp (L_1 \|h_{k} - h_0^*\|).
\end{align}
We can get the following bound for $\|F(h_{k-1})\|$, for all $k \ge 1$:
\begin{align}\label{eq-korpelevich-9}
\|F(h_{k})\| &\leq \|F(h_{k}) - F(h^*_0)\| + \|F(h^*_0)\| \cr
&\leq \|h_k - h_0^*\| (L_0+ L_1 \|F(h_0^*)\|) \exp (L_1 \|h_{k} - h_0^*\|)+  \|F(h^*_0)\| \cr
\end{align}
Earlier, in equation \eqref{eq-korpelevich-4} we proved that for arbitrary solution $u^*$ the following bound hold for any $k \geq 0$ 
\[\|h_{k} - u^*\|^2 \leq  \|h_0 - u^*\|^2.\]

Let $R_0 =  \|h_0 - h_0^*\|^2$, 
Then, it follows that for all $k \geq 1$:
\[ \|h_{k} - h_0^*\| \leq \sqrt{ R_0}.\]

Using these facts and equation~\eqref{eq-korpelevich-8} for all $k \ge 1$:

\begin{align}\label{eq-korpelevich-10}
\|F(h_{k})\| &\leq  \|h_k - h_0^*\| (L_0+ L_1 \|F(h_0^*)\|) \exp (L_1 \|h_{k} - h_0^*\|) + \|F(h_0^*)\| \cr
&\leq \sqrt{R_0} (L_0+ L_1 D_2) \exp (L_1 \sqrt{ R_0}) +  D_2 .
\end{align}
We showed that the sequence \{$\|F(h_k)\|$\} is upper bounded by some constant $\overline{C}_1$, where $\overline{C}_1 = \sqrt{R_1} (L_0+ L_1 D_2) \exp (L_1 \sqrt{ R_0}) +  D_2 $. Then, we conclude that 
\begin{equation}
\begin{aligned}
\label{eq-korpelevich-11}
\g_k &= \min \{ \frac{1}{4\mu} , \frac{1}{2 \sqrt{2} e L_0}, \frac{1}{2 \sqrt{2} e L_1 \|F(h_{k})\| } \} \cr
&\geq \min \{ \frac{1}{4\mu} , \frac{1}{2 \sqrt{2} e L_0}, \frac{1}{2 \sqrt{2} e L_1 \overline{C}_1 } \} 
\end{aligned}
\end{equation}

For both cases $\a \in (0,1)$ and $\alpha=1$ in equations~\eqref{eq-korpelevich-6}, ~\eqref{eq-korpelevich-11} we showed that stepsizes sequence $\{\gamma_k\}$ is lower bounded by some constant $\underline{\gamma}$. Where 
\[\underline{\gamma} = \min \{ \frac{1}{4\mu} ,\frac{1}{3 \sqrt{2} K_0}, \frac{1}{C_1}, \frac{1}{ 3 \sqrt{2} K_1 C_1^{\a}}, \frac{1}{ 3 \sqrt{2} K_2 }\} \text{ for } \a \in (0,1),\]  
\[\underline{\gamma} = \min \{ \frac{1}{4\mu} , \frac{1}{\overline{C}_1}, \frac{1}{2 \sqrt{2} e L_0}, \frac{1}{2 \sqrt{2} e L_1 \overline{C}_1 } \} \text{ for } \a=1, \]
i.e. $\gamma_k \geq \underline{\gamma}$ for all $k \geq 0$.

\textbf{(c)} Based on the results from part (b), for both cases ($\a \in (0,1)$ and $\a=1$) there exists some positive constant $\underline{\gamma}$ such that $\underline{\gamma} \leq \gamma_k$ for all $k \geq 0$. Using this fact and equation~\eqref{eq-korpelevich-4}, we obtain for any $k\geq 1$:

\begin{equation}
\begin{aligned}
\label{eq-korpelevich-11}
 \|h_{k+1}  - u^*\|^2 &\leq \|h_k - u^*\|^2  -\frac{1}{2}\|u_k - h_k\|^2 - 2 \underline{\gamma} \mu \dist^p(u_k, U^*).
\end{aligned}
\end{equation}

The equation~\eqref{eq-korpelevich-11} satisfies the conditions of Lemma~\ref{lemma-polyak11-det} with 
\begin{align}\bar v_{k} = \|h_{k}  - h^*\|^2, \quad  \bar a_k = 0, \cr
\bar z_k =  \frac{1}{2}\|u_k - h_{k}\|^2 +  2 \mu \underline{\gamma} \, \dist^p(u_{k}, U^*), \quad \bar b_k = 0 .
\end{align}

By Lemma~\ref{lemma-polyak11}, it follows that the sequence $\{\bar v_k\}$ converges to a non-negative scalar for any $u^*\in U^*$, and we have
\[
\sum_{k=0}^{\infty}  \dist^p(u_k, U^*) < \infty,\quad \sum_{k=0}^{\infty} \|u_k - h_k\|^2 < \infty.\]
Thus, it follows that 
\begin{equation}\label{eq-korp-an0-det-a1}
\lim_{k\to\infty}\dist^p(u_k, U^*) =0 .
\end{equation}
\begin{equation}\label{eq-korp-an1-det-a1}
\lim_{k\to\infty}\|u_k - h_k\| = 0 .
\end{equation}
Since $\bar v_k$ converges for any given $u^* \in U^*$  we can conclude that $\|h_k - u^*\|$ converges for all $u^* \in U^*$. Therefore, the sequence $\{h_k\}$ is bounded and has accumulation points.
In view of relation~\eqref{eq-korp-an1-det-a1},
the sequences $\{u_k\}$ and $\{h_k\}$ have the same accumulation points.
Then, there exists a convergent subsequnce $\{h_{k_{i}}\}$, such that $\bar{u} \in U$ is its limit point, i.e.,
\begin{equation}\label{eq-korp-an3-det-a1}
\lim_{i\to\infty} \|h_{k_i}-\bar u\| = 0.
\end{equation}
By relation~\eqref{eq-korp-an1-det-a1}, it follows that
\[\lim_{i\to\infty} \|u_{k_i}-\bar u\| = 0. \]
By the continuity of the distance function $\dist(\cdot,U^*)$, from relation~\eqref{eq-korp-an0-det-a1} we conclude that $\dist(\bar u,U^*)=0$ , which implies that $\bar u\in U^*$ since the set $U^*$ is closed.
 Since the sequence $\{\|h_k - u^*\|^2\}$ 
converges for any $u^*\in U^*$, it follows that 
$\{\|h_k - \bar u\|^2\}$ 
converges , and by relation~\eqref{eq-korp-an3-det-a1} we conclude that 
$\lim_{k\to\infty}\|h_k - \bar u\|^2=0$.

\end{proof}

\textbf{Proof of Theorem~\ref{thm-Korp-rates}}
\label{App-Korp-rates}
\begin{proof}
Since the solution set $U^*$ is closed, there exists a projection $h_k^*$ of the iterate $h_k$ on the set $U^*$, i.e., there exists a point $h_k^*\in U^*$ such that $\|h_k- h_k^*\|= \dist(h_k, U^*)$. Thus, by letting in $u^*=h_k^*$ in~\eqref{eq-korpelevich-thm-asym} (the result of Theorem~\ref{thm-Korpelevich-asym}(a) ), we obtain for all $k\ge 0$,
\begin{equation}
    \begin{aligned}
    \label{eq-korp-rates-1}
     \|h_{k+1} - h_k^*\|^2  &\leq \|h_k  - h_k^*\|^2  - \frac{1}{2}\|u_k - h_{k}\|^2 - 2 \mu \gamma_k \, \dist^p(u_{k}, U^*).
    \end{aligned}
\end{equation}

In view of $\|h_k- h_k^*\|= \dist(h_k, U^*)$ and 
$\dist(h_{k+1},U^*)\le \|h_{k+1} - h_k^*\|$, it follows that for all $k\ge1$,
\begin{equation}
    \begin{aligned}
    \label{eq-korp-rates-2}
     \dist^2(h_{k+1}, U^*) &\leq \dist^2(h_k, U^*)  - \frac{1}{2}\|u_k - h_{k}\|^2 -  2 \mu \gamma_k \, \dist^p(u_{k}, U^*).
    \end{aligned}
\end{equation}
Next, we estimate the term $\dist^p(u_k,U^*)$ in~\eqref{eq-korp-rates-2}. By the triangle inequality, we have
\[\|h_k - u^*\| \leq \|u_k - h_k\| + \|u_k - u^*\| \qquad \hbox{for all }u^*\in U^*,\]
and by taking the minimum over $u^*\in U^*$ on both sides of the preceding relation, we obtain
\begin{equation}\label{eq-korp-rates-3}
\dist(h_k,U^*)\le \|u_k - h_k\| +\dist(u_k,U^*).
\end{equation}

\textbf{Case (a) $p = 2$}
Squaring both sides of inequality~\eqref{eq-korp-rates-3}, and using $(\sum_{i=1}^m a_i)^2\le m\sum_{i=1}^m a_i^2$, which is valid for any scalars $a_i$, $i=1,\ldots,m,$ and any integer $m\ge 1$, we obtain
\begin{equation}\label{eq-korp-rates-4}
- 2 \, \dist^2(u_k, U^*) 
\leq 2 \|u_k - h_k\|^2 - \, \dist^2(h_k, U^*).
\end{equation}

Combining this inequality with \eqref{eq-korp-rates-2}, we get that for any $k \geq 0$:
\begin{equation}
    \begin{aligned}
    \label{eq-korp-rates-5}
     \dist^2(h_{k+1}, U^*) &\leq (1 - \mu \gamma_k) \dist^2(h_k, U^*)  - \frac{1}{2} \|u_k - h_{k}\|^2 + 2 \mu \gamma_k \, \|u_k -  h_{k}\|^2 .
    \end{aligned}
\end{equation}
Based on stepsizes choice $\g_k \mu \leq \frac{1}{4}$ and based on the result of Theorem~\ref{thm-Korpelevich-asym} (b) the stepsizes sequence 
$\{\g_k \}$ is lower bounded,  i.e., $\gamma_k \geq \underline{\gamma}$ for all $k\ge0$. Then, for all $k \geq 0$:
\begin{equation}
    \begin{aligned}
    \label{eq-korp-rates-6}
     \dist^2(h_{k+1}, U^*) &\leq (1 - \mu \underline{\gamma}) \dist^2(h_k, U^*) .
    \end{aligned}
\end{equation}

Thus, by defining $D_{k}^2 = \dist^2(h_k, U^*)$  we obtain that for all $k \ge 0$,
\begin{equation}
    \begin{aligned}
    \label{eq-korp-rates-7}
     D_{k+1}^2 & \leq (1 - \mu \underline{\gamma}) D_k^2.
    \end{aligned}
\end{equation}

\textbf{Case (b)  $p>2$ } 
By using Lemma~\ref{inequality-hoed-2},
we further obtain
\begin{equation}
\begin{aligned}
\label{eq-korp-rates-8-0}
\dist^p(h_k,U^*) &\le (\|u_k - h_k\| +\dist(u_k,U^*))^p \cr
&\le 2^{p-1} \|u_k - h_k\|^p + 2^{p-1}   \, \dist^p(u_k,U^*).
\end{aligned}
\end{equation}
Using projection inequality~\eqref{eq-proj3}, and stepsizes choice  we obtain:
\[ \|u_k - h_k\| \leq \|\gamma_{k} F(h_{k})\| \leq 1 . \]
Combining this result with equation~\eqref{eq-korp-rates-8-0} we get:
\begin{equation}
\begin{aligned}
\label{eq-korp-rates-8-1}
\dist^p(h_k,U^*) &\le 2^{p-1} \|u_k - h_k\|^p + 2^{p-1} \dist^p(u_k,U^*) \cr
&\le 2^{p-1} \|u_k - h_k\|^2 + 2^{p-1}   \, \dist^p(u_k,U^*).
\end{aligned}
\end{equation}
And by rearranging terms, we obtain the following:
\begin{equation}
\begin{aligned}
\label{eq-korp-rates-8-2}
- \dist^p(u_k,U^*) &\leq \|u_k - h_k\|^2 -  2^{-p + 1}  \, \dist^p(h_k,U^*) .
\end{aligned}
\end{equation}
Combining this inequality with \eqref{eq-korp-rates-2},
we obtain that for any $k \geq 1$:
\begin{equation}
    \begin{aligned}
    \label{eq-korp-rates-8}
     \dist^2(u_{k+1}, U^*) &\leq \dist^2(u_k, U^*)  - (\frac{1}{2} - 2 \mu \gamma_k) \|u_k - h_{k}\|^2 -  2^{-p+2} \gamma_k \mu \, \dist^p(h_k,U^*).
    \end{aligned}
\end{equation}

By the stepsizes choice $\gamma_k \leq \frac{1}{4 \mu}$ then $(\frac{1}{2} - 2 \mu \gamma_k) \|u_k - h_{k}\|^2 \leq 0$, and $\gamma_k \geq \underline{\gamma}$ for all $k \geq 0$. Thus, by defining $D_{k}^2 = \dist^2(h_k, U^*)$  we obtain that for all $k \ge 0$,
\begin{equation}
    \begin{aligned}
    \label{eq-korp-rates-9}
     D_{k+1}^2 &\leq D_k^2-  2^{-p+2} \underline{\gamma}\mu \, (D_k^2)^{p/2} .
    \end{aligned}
\end{equation}

Using Lemma~\ref{lemma-polyak-6} with $x_k = D_k^2$, $q = \frac{p-2}{2}$, $a_k = 2^{-p +2} \underline{\gamma} $, we obtain, for all $k\geq 0$:
\begin{equation}
    \begin{aligned}
    \label{eq-korp-rates-10}
     D_{k+1}^2 & \leq \frac{D_0^2}{ (1 + (p-2) \underline{\gamma} 2^{-p+1}   (D_0^2)^{(p-2)/2} (k+1))^{2/(p-2)}}.
    \end{aligned}
\end{equation}

\end{proof}

\section{Popov Method Analysis}
\label{App-Popov}
Proof of Lemma~\ref{Lemma-Popov}
\begin{proof}
Let $k\ge 1$ be arbitrary but fixed. From the definition of $u_{k+1}$ in~\eqref{eq-popov-det}, we have 
$\|u_{k+1} - y\|^2 = \|P_{U} (u_{k} - \g_k  F(h_k)) - y\|$ for any $y \in U$. Using the projection inequality
~\eqref{eq-proj3} of Lemma~\ref{lem-proj}
we obtain for any $y\in U$,
\begin{equation}
\begin{aligned}
\label{eq-lemmaproof1_1}
\|u_{k+1} - y\|^2 &\leq \|u_k - \g_k F(h_{k}) - y\|^2 - \|u_{k} - \g_k F(h_{k}) - u_{k+1}\|^2 \cr
&= \|u_k  - y\|^2 - \|u_{k+1} - u_k\|^2 - 2\g_k \langle F(h_{k}), u_{k+1} - y\rangle.
\end{aligned}
\end{equation}
We next consider the term $\|u_{k+1} - u_k\|^2$, where we add 
and subtract $h_k$, and thus obtain
\begin{equation}
\begin{aligned}
\label{eq-lemmaproof1_2}
\|u_{k+1} -u_k\|^2 = &\|(u_{k+1} -h_k) - (u_k - h_k)\|^2  \\
= &\|u_{k+1} -h_k\|^2 + \|u_k - h_k\|^2 - 2\langle u_k - h_k, u_{k+1} - h_k \rangle \\
=  &\|u_{k+1} -h_k\|^2 + \|u_k - h_k\|^2 - 2\langle u_k - \g_k F(h_{k-1}) - h_k, u_{k+1} - h_k \rangle \\
&- 2 \g_k \langle F(h_{k-1}), u_{k+1} - h_k\rangle,
\end{aligned}
\end{equation}
where the last equality is obtained by
adding and subtracting $2 \g_k \langle F(h_{k-1}, ), u_{k+1} - h_k\rangle$. Next, we use  another projection property \eqref{eq-proj1} with $v = u_k - \g_k F(h_{k-1}), u = u_{k+1}$ and obtain: 
\begin{equation}
\begin{aligned}
\langle u_k - \g_k F(h_{k-1}) - h_k, u_{k+1} - h_k \rangle \leq 0.
\end{aligned}
\end{equation}
Therefore,
\begin{equation}
\begin{aligned}
\label{eq-lemmaproof1_3}
\|u_{k+1} - u_k\|^2 &\ge \|u_{k+1} -h_k\|^2 + \|u_k - h_k\|^2 - 2 \g_k \langle F(h_{k-1}), u_{k+1} - h_k\rangle .
\end{aligned}
\end{equation}

Combining \eqref{eq-lemmaproof1_1} and \eqref{eq-lemmaproof1_3} 
we can see that for any $y\in U$,
\begin{equation}
\begin{aligned}
\label{eq-lemmaproof1_4}
\|u_{k+1} - y\|^2 \leq &\|u_k  - y\|^2 - \|u_{k+1} -h_k\|^2 - \|u_k - h_k\|^2  - 2 \g_{k} \langle F(h_{k}), u_{k+1}-h_k\rangle \cr
&\  
- 2 \g_{k} \langle F(h_{k}), h_k-y\rangle
+ 2 \g_{k} \langle  F(h_{k-1}), u_{k+1}-h_k\rangle\cr
= &\|u_k  - y\|^2 - \|u_{k+1} -h_k\|^2 - \|u_k - h_k\|^2
- 2 \g_{k} \langle F(h_{k}), h_k-y\rangle\cr
&\ + 2 \g_{k} \langle F(h_{k-1}) - F(h_{k}), u_{k+1}-h_k\rangle.  
\end{aligned}
\end{equation}
To estimate the last inner product in~\eqref{eq-lemmaproof1_4}, we write
\[\langle F(h_{k-1}) - F(h_{k}), u_{k+1}-h_k\rangle \le \|F(h_{k-1}) - F(h_{k})\|\, \|u_{k+1}-h_k\|. \]
From the definitions of $u_{k+1}$ and $h_{k+1}$ in~\eqref{eq-popov-det}, we have 
$u_{k+1} =P_{U} (u_{k} - \g_k  F(h_k))$ and 
$h_k=P_U(u_k-\g_k F(h_{k-1}))$. 
Thus, by using the Lipschitz continuity of the projection operator (see equation \eqref{eq-proj5} in Lemma 1), 
we obtain $\|u_{k+1}-h_k\|\le \g_k\|F(h_{k-1})-F(h_k)\|$, implying that
\[\langle F(h_{k-1}) - F(h_{k}), u_{k+1}-h_k\rangle \le \g_k\|F(h_{k-1}) - F(h_{k})\|^2.\]
Upon substituting the preceding estimate back in relation~\eqref{eq-lemmaproof1_4},
we obtain the desired relation:
\begin{equation}
\begin{aligned}
\label{eq-lemmaproof1_5}
\|u_{k+1} - y\|^2 \leq &\|u_k  - y\|^2 - \|u_{k+1} -h_k\|^2 - \|u_k - h_k\|^2  - 2 \g_{k} \langle F(h_{k}), h_k-y\rangle \\
&+ 2 \g_{k}^2 \|F(h_{k}) -  F(h_{k-1})\|^2.
\end{aligned}
\end{equation}

\end{proof}

\textbf{Proof of Theorem~\ref{thm-Popov-a01-asymp}}
\label{App-Popov-asym-01}
\begin{proof}
By Lemma~\ref{Lemma-Popov} we have that the following inequality holds for all $k\ge1$ and all $y \in U$,
\begin{equation}
\begin{aligned}
\label{eq-popov-est-0}
\|u_{k+1} - y\|^2 \leq &\|u_k  - y\|^2 - \|u_{k+1} -h_k\|^2 - \|u_k - h_k\|^2  - 2 \g_{k} \langle F(h_{k}), h_k-y\rangle \\
&+ 2 \g_{k}^2 \|F(h_{k}) -  F(h_{k-1})\|^2.
\end{aligned}
\end{equation}
We want to estimate the last term on the LHS of the inequality~\eqref{eq-popov-est-0} using the fact that the operator $F(\cdot)$ is an $\a$-symmetric operator. 

\textbf{Case (I) $\a \in (0,1)$}

Using the definition of $\a$-symmetric operator (as given in~\eqref{alpha-property}), we obtain that for all $k\ge1$,
\begin{equation}
\begin{aligned}
\label{eq-popov-est-1}
\|F(h_{k}) -  F(h_{k-1})\|^2 &\leq \|h_k - h_{k-1}\|^2 (K_0 + K_1 \|F(h_{k-1})\|^{\alpha} + K_2 \|h_{k} - h_{k-1}\|^{\a / (1 - \a)})^2 \cr
&\leq 3 K_0^2 \|h_k - h_{k-1}\|^2 + 3 K_1^2\|F(h_{k-1})\|^{2\a} \|h_k - h_{k-1}\|^2  \cr
&\ \ + 3K_2^2 \| h_{k} - h_{k-1}\|^{2\a /(1 - \a)} \| h_{k} - h_{k-1}\|^2,
\end{aligned}
\end{equation}
where the last inequality is obtained by using relation $(\sum_{i=1}^m a_i)^2\le m\sum_{i=1}^m a_i^2$.

Using the definition of the iterate updates \eqref{eq-popov-det}, the projection property \eqref{eq-proj3} for $v=h_{k-1} - \g_k F(h_{k-1}), u = h_{k-1}$ and triangle inequality we get: 
\[\| h_{k} - h_{k-1}\| \leq \|u_k -h_{k-1} - \g_k F(h_{k-1})\| \leq \|u_k -h_{k-1} \| + \g_k\| F(h_{k-1})\|. \]

By the definition of the stepsize $\g_k$, we have that $\g_k \leq \frac{1}{\|F(h_{k-1})\|}$. Therefore, it follows that
\begin{equation}
\begin{aligned}
\| h_{k} - h_{k-1}\| \leq \|u_k -h_{k-1} \| + \g_k \|F(h_{k-1})\| 
\le \|u_k -h_{k-1} \| + 1.
\end{aligned}
\end{equation}
Combining this inequality with~\eqref{eq-popov-est-1} we obtain:
\begin{equation}
\begin{aligned}
\label{eq-popov-est-2}
\|F(h_{k}) -  F(h_{k-1})\|^2  &\leq 3 K_0^2 \|h_k - h_{k-1}\|^2 + 3 K_1^2\|F(h_{k-1})\|^{2 \a} \|h_k - h_{k-1}\|^2  \cr
&+ 3K_2^2 ( \| u_k - h_{k-1}\| + 1)^{2\a /(1 - \a)} \| h_{k} - h_{k-1}\|^2.
\end{aligned}
\end{equation}

By stepsize choice, $\g_k \leq \frac{1}{6 \sqrt{2} K_0}$, $\g_k \leq \frac{1}{6 \sqrt{2}K_1\|F(h_{k-1})\|^{\a}}$, and $\g_k \leq \frac{1}{6 \sqrt{2} K_2^2 ( \| u_k - h_{k-1}\| + 1)^{\a /(1 - \a)}}$, then we obtain: 

\begin{equation}
\begin{aligned}
\label{eq-popov-est-3}
2 \g_k^2 \|F(h_{k}) -  F(h_{k-1})\|^2  &\leq \frac{1}{12} \|h_k - h_{k-1}\|^2 + \frac{1}{12} \|h_k - h_{k-1}\|^2  \cr
&+ \frac{1}{12} \| h_{k} - h_{k-1}\|^2 \cr
&= \frac{1}{4} \|h_k - h_{k-1}\|^2.
\end{aligned}
\end{equation}

\textbf{Case (II) $\alpha = 1$}

Using the definition of $\a$-symmetric operator (as given in~\eqref{alpha-property-1}), we obtain that for all $k\ge1$,
\begin{equation}
\begin{aligned}
\label{eq-popov-a1-est-1}
\|F(h_{k}) -  F(h_{k-1})\| &\leq \|h_k - h_{k-1}\| (L_0+ L_1 \|F(h_{k})\|) \exp (L_1 \|h_{k} - h_{k-1}\|).
\end{aligned}
\end{equation}
By the projection property and triangle inequality:
\[\| h_{k} - h_{k-1}\| \leq \|u_k -h_{k-1} - \g_k F(h_{k-1})\| \leq \|u_k -h_{k-1} \| + \g_k \| F(h_{k-1})\|. \]
By the definition of the stepsize $\g_k$, we have that $\g_k \leq \frac{1}{\|F(h_{k-1})\|}$. Therefore, it follows that
\[\| h_{k} - h_{k-1}\| \leq \|u_k -h_{k-1} \| + 1. \]
Then:

\begin{equation}
\begin{aligned}
\label{eq-popov-a1-est-2}
\|F(h_{k}) -  F(h_{k-1})\| &\leq \|h_k - h_{k-1}\| (L_0+ L_1 \|F(h_{k-1})\|) \exp (L_1 \|u_k - h_{k-1}\| + 1).
\end{aligned}
\end{equation}

Based on the step-size choice, $\g_k \leq \frac{1}{ 2 \sqrt{2} (L_0 + L_1 \|F(h_{k-1})\|) \exp (L_1 \|u_k - h_{k-1}\| + 1)}$, we have:
\begin{equation}
\begin{aligned}
\label{eq-popov-a1-est-3}
2 \gamma_k^2 \|F(h_{k}) -  F(h_{k-1})\|^2 &\leq \frac{1}{4}\|h_k - h_{k-1}\|^2.
\end{aligned}
\end{equation}

For both cases quantity $2 \gamma_k^2 \|F(h_{k}) -  F(h_{k-1})\|^2 \leq \frac{1}{4}\|h_k - h_{k-1}\|^2$ is upperbounded for all $k \geq 1$. Then combining upperbounds in equations~\eqref{eq-popov-est-3}, and~\eqref{eq-popov-a1-est-3} with equation~\eqref{eq-popov-est-0} we obtain that for any $k\ge 1$ the next inequality holds:
\begin{equation}
    \begin{aligned}
    \label{eq-popov-basic-2}
     \|u_{k+1}  - y\|^2 & \leq \|u_k  - y\|^2 - \|u_{k+1} -h_k\|^2 - \|u_k - h_k\|^2  \cr
     &- 2 \g_{k} \langle F(h_k), h_k - y\rangle + \frac{1}{4}\|h_k - h_{k-1}\|^2.
    \end{aligned}
\end{equation}
Now using the triangle inequality and relation $(\sum_{i=1}^m a_i)^2\le m\sum_{i=1}^m a_i^2$ we obtain:
\[\|h_{k} - h_{k-1}\|^2 \leq 2 \|u_k - h_k\|^2 + 2 \|u_{k} - h_{k-1}\|^2 . \]
Using the preceding etimate in equation~\eqref{eq-popov-basic-2} we get:
\begin{equation}
    \begin{aligned}
    \label{eq-popov-basic-3}
     \|u_{k+1}  - y\|^2 + \|u_{k+1} -h_k\|^2  & \leq \|u_k  - y\|^2   + \frac{1}{2}\|u_k - h_{k-1}\|^2\cr
     &- \frac{1}{2} \|u_k - h_k\|^2  - 2 \g_{k} \langle F(h_k), h_k - y\rangle.
    \end{aligned}
\end{equation}

Now we plug $y=u^*$ where $u^*$ is an arbitrary solution and use the assumption in the theorem on quasi sharpness of the operator. Doing so we obtain the following recursive inequality for any $u^* \in U^*$ and $k \geq 1$

\begin{equation}
    \begin{aligned}
    \label{eq-popov-basic-4}
     \|u_{k+1}  - u^*\|^2 + \|u_{k+1} -h_k\|^2  & \leq \|u_k  - u^*\|^2   + \frac{1}{2}\|u_k - h_{k-1}\|^2 \cr
     &- \frac{1}{2} \|u_k - h_k\|^2  - 2 \g_{k} \langle F(h_k), h_k - u^*\rangle \cr
     &\leq \|u_k  - u^*\|^2   + \|u_k - h_{k-1}\|^2 \cr
      &- \frac{1}{2}\|u_k - h_{k-1}\|^2 - \frac{1}{2}\|u_k - h_{k}\|^2 .
    \end{aligned}
\end{equation}
The first inequality holds because $\langle F(h_k), h_k - u^*\rangle$ is positive for all $u^* \in U^*$ when $F(\cdot)$ is $p$-quasi sharp. We conclude that for any $u^* \in U^*$, the  norms $\|u_k - u^*\|$ and $\|u_{k+1} - h_k\|$ are bounded.


(b) Now, we want to estimate $\|F(h_{k-1})\|$, since this term is present in the denominator of the stepsizes. Let $u^*_1 = P_{U^*}(u_1)$ be a projection of $u_1$ onto the solution set $U^*$. We add and subtract $F(u^*_1)$ , and get $\|F(h_{k-1})\| = \|F(h_{k-1}) - F(u^*_1) + F(u^*_1)\| \leq \|F(h_{k-1}) - F(u^*_1)\| + \|F(u^*_1)\|$. 
Now, we estimate the first term using the assumption on the operator.

\textbf{Case (I) $\alpha \in (0,1)$}
Using the definition of $\a$-symmetric operator (as given in~\eqref{alpha-property}), we obtain that for all $k\ge1$,
\begin{align}\label{eq-popov-basic-5}
\|F(h_{k-1}) - F(u^*_1)\| \leq (K_0 + K_1 \|F(u^*_1)\|^{\a})\|h_{k-1} - u^*_1\| + K_2 \|h_{k-1} - u^*_1\|^{1/(1 - \alpha)} .
\end{align}
we can get the following bound for $\|F(h_{k-1})\|$, for all $k \ge 1$:
\begin{align}\label{eq-popov-basic-6}
\|F(h_{k-1})\| &\leq \|F(h_{k-1}) - F(u^*_1)\| + \|F(u^*_1)\| \cr
&\leq (K_0 + K_1 \|F(u^*_1)\|^{\a})\|h_{k-1} - u^*_1\| + K_2 \|h_{k-1} - u^*_1\|^{1/(1 - \alpha)} +  \|F(u^*_1)\| . \cr
\end{align}
Earlier, in equation \eqref{eq-popov-basic-4} we proved that for arbitrary solution $u^*$ the following bounds hold for any $k \leq 0$ 
\[\|u_{k+1} - u^*\|^2 + \|u_{k+1} - h_k\|^2  \leq \|u_1 - u^*\|^2 + \|u_1 - h_0\|^2.\]
Let $R_1^2 = \|u_1 - u^*_1\|^2 + \|u_1 - h_0\|^2$, then  
\[(\|u_k - u^*_1\| + \|u_{k} - h_{k-1}\|)^2 \leq 2 \|u_k - u^*_1\|^2 + 2 \|u_{k} - h_{k-1}\|^2 \leq 2 R_1^2.\]
Therefore, it follows that for all $k \geq 1$
\[\|h_{k-1} - u_1^*\| \leq \|u_{k} - u^*_1\| + \|u_{k} - h_{k-1}\| \leq \sqrt{ 2} R_1,\]
\[\|u_{k} - h_{k-1}\| \leq  \sqrt{2} R_1.\]

Then using the facts above and~\eqref{eq-popov-basic-6} for all $k \ge 1$:

\begin{align}\label{eq-popov-basic-7}
\|F(h_{k-1})\| &\leq  (K_0 + K_1 \|F(u_1^*)\|^{\a}) \sqrt{ 2} R_1 + K_2 (\sqrt{ 2} R_1)^{1/(1 - \alpha)} + \|F(u_1^*)\| .\cr
\end{align}

Now, we define a new constant $C_1 = (K_0 + K_1 \|F(u_1^*)\|^{\a})\sqrt{ 2} R_1 + K_2 (\sqrt{2} R_1)^{1/(1 - \alpha)} +  \|F(u_1^*)\| $, then $\|F(h_{k-1})\| \leq C_1$ for all $k \geq 1$. Since function $(x)^s$ is monotone for any $x>0$ and any $s > 0$, then $\|F(h_{k-1})\|^{\a /2} \leq C_1^{\a / 2}$ and $(\|u_k - h_{k-1}\| + 1)^{\a /(1 - \a)} \leq (\sqrt{ 2} R_1 + 1)^{\a /(1 - \a)}$. Using these facts, we conclude that 
\begin{equation}
\begin{aligned} 
\label{eq-popov-basic-7-1}
\g_k &= \min \{  \frac{1}{4\mu} ,\frac{1}{\|F(h_{k-1})\|}, \frac{1}{6\sqrt{2} K_0}, \frac{1}{6 \sqrt{2} K_1 \|F(h_{k-1})\|^{\a / 2}}, \frac{1}{6 \sqrt{2} K_2 (\|u_k - h_{k-1}\| + 1)^{\a /(1 - \a)}} \} \cr
&\geq \min \{ \frac{1}{4 \mu}, \frac{1}{C_1}, \frac{1}{6\sqrt{2} K_0}, \frac{1}{6 \sqrt{2} K_1  C_1^{\a / 2}}, \frac{1}{6 \sqrt{2} K_2 (\sqrt{2}R_1 + 1)^{\a /(1 - \a)}}\} .
\end{aligned}    
\end{equation}


\textbf{Case (II) $\alpha = 1$}
Using the definition of $\a$-symmetric operator (as given in~\eqref{alpha-property-1}), we obtain that for all $k\ge1$,
\begin{align}\label{eq-popov-a1-est-7}
\|F(h_{k}) -  F(u_1^*)\| &\leq \|h_k - u_1^*\| (L_0+ L_1 \|F(u_1^*)\|) \exp (L_1 \|h_{k} - u_1^*\|).
\end{align}
Hence, we obtain the following bound for $\|F(h_{k-1})\|$, for all $k \ge 1$:
\begin{align}\label{eq-popov-a1-est-8}
\|F(h_{k-1})\| &\leq \|F(h_{k-1}) - F(u^*_1)\| + \|F(u^*_1)\| \cr
&\leq \|h_k - u_1^*\| (L_0+ L_1 \|F(u_1^*)\|) \exp (L_1 \|h_{k} - u_1^*\|)+  \|F(u^*_1)\| \cr
\end{align}
Earlier, in equation \eqref{eq-popov-basic-4} we proved that for arbitrary solution $u^*$ the following bounds hold for any $k \geq 0$ 
\[\|u_{k+1} - u^*\|^2 + \|u_{k+1} - h_k\|^2  \leq \|u_1 - u^*\|^2 + \|u_1 - h_0\|^2.\]

Let $R_1^2 = \|u_1 - u^*_1\|^2 + \|u_1 - h_0\|^2$, then  
\[ (\|u_k - u^*_1\|+ \|u_{k} - h_{k-1}\|)^2 \leq 2 \|u_k - u^*_1\|^2 + 2 \|u_{k} - h_{k-1}\|^2 \leq 2 R_1^2 .\]
Therefore, it follows that for all $k \geq 1$:
\[ \|h_{k-1} - u_1^*\| \leq \|u_{k} - u^*_1\| + \|h_{k-1} - u_k\| \leq \sqrt{2} R_1,\]
\[\|u_k - h_{k-1}\| \leq \sqrt{2} R_1.\]

Then using the preceding relation and equation~\eqref{eq-popov-a1-est-8} for all $k \ge 1$:

\begin{align}\label{eq-popov-a1-est-9}
\|F(h_{k-1})\| &\leq  \|h_k - u_1^*\| (L_0+ L_1 \|F(u_1^*)\|) \exp (L_1 \|h_{k} - u_1^*\|) + \|F(u_1^*)\| \cr
&\leq \sqrt{2} R_1 (L_0+ L_1 \|F(u^*_1)\|) \exp (L_1 \sqrt{2} R_1) +  \|F(u^*_1)\|.\cr
\end{align}

Now, we define a new constant $\overline{C}_1 = \sqrt{2} R_1 (L_0+ L_1 \|F(u^*_1)\|) \exp (L_1 \sqrt{2} R_1) +  \|F(u^*_1)\| $, then $\|F(h_{k-1})\| \leq \overline{C}_1$ for all $k \geq 1$. And since function $\exp (x)$ is monotone for any $x>0$, then  $\exp (L_1\|u_k - h_{k-1}\| + 1) \leq \exp (L_1 \sqrt{2} R_1 + 1)$. Using these facts, we conclude that
\begin{equation}
\begin{aligned} 
\label{eq-popov-a1-est-9-1}
\g_k &= \min \{ \frac{1}{4\mu} , \frac{1}{\|F(h_{k-1})\|}, \frac{1}{4 (L_0 + L_1 \|F(h_{k-1})\|) \exp (L_1 \|u_k - h_{k-1}\| + 1)} \}  \cr
&\geq \min \{ \frac{1}{4\mu} , \frac{1}{\overline{C}_1}, \frac{1}{4 (L_0 + L_1 \overline{C}_1) \exp (L_1 \sqrt{2} R_1 + 1)} \} 
\end{aligned}
\end{equation}


For both cases $\alpha \in (0,1)$ and $\alpha=1$ in equations \eqref{eq-popov-basic-7-1}, \eqref{eq-popov-a1-est-9-1} we showed that stepsizes sequence $\{\gamma_k \}$ is lower bounded by constant $\underline{\gamma}$, where

\begin{equation}
\begin{aligned}
   \underline{\gamma} = \min \{ \frac{1}{4 \mu}, \frac{1}{C_1}, \frac{1}{6\sqrt{2} K_0}, \frac{1}{6 \sqrt{2} K_1 (C_1)^{\a / 2}}, \frac{1}{6 \sqrt{2} K_2 (\sqrt{2}R_1 + 1)^{\a /(1 - \a)}}\} \text{ for } \alpha \in (0,1) \cr
    \underline{\gamma} = \min \{ \frac{1}{4\mu} , \frac{1}{\overline{C}_1}, \frac{1}{4 (L_0 + L_1 \overline{C}_1) \exp (L_1 \sqrt{2} R_1 + 1)} \} \text{ for } \alpha = 1 
\end{aligned}
\end{equation}

(c) Based on the results from part (b)  $\underline{\gamma} \leq \gamma_k$ for all $k \geq 1$. Then using $p$-quasi-sharpness of the operator and equation~\eqref{eq-popov-basic-3} with $y=u^*$, where $u^*$ is an arbitrary solution, we obtain for any $k\geq 1$:

\begin{equation}
    \begin{aligned}
    \label{eq-popov-basic-8}
     \|u_{k+1}  - u^*\|^2 + \|u_{k+1} -h_k\|^2  & \leq \|u_k  - u^*\|^2   + \frac{1}{2}\|u_k - h_{k-1}\|^2.\cr
     &- \frac{1}{2} \|u_k - h_k\|^2  - 2 \g_{k} \langle F(h_k), h_k - u^*\rangle \cr
     &\leq \|u_k  - u^*\|^2   + \|u_k - h_{k-1}\|^2 \cr
      &- \frac{1}{2}\|u_k - h_{k}\|^2 - \frac{1}{2}\|u_k - h_{k-1}\|^2-  2 \g_k \mu \, \dist^p(h_{k}, U^*) \cr
     &\leq \|u_k  - u^*\|^2   + \|u_k - h_{k-1}\|^2 \cr
      &- \frac{1}{2}\|u_k - h_{k}\|^2 - \frac{1}{2}\|u_k - h_{k-1}\|^2-  2 \underline{\gamma} \mu \, \dist^p(h_{k}, U^*).
    \end{aligned}
\end{equation}


The equation~\eqref{eq-popov-basic-8} satisfies the condition of Lemma~\ref{lemma-polyak11-det} with 
\begin{align}\bar v_{k} = \|u_{k}  - u^*\|^2 + \|u_{k} -h_{k-1}\|^2, \quad  \bar a_k = 0, \cr
\bar z_k = \frac{1}{2}\|u_k - h_{k-1}\|^2 + \frac{1}{2}\|u_k - h_{k}\|^2 +  2 \mu \underline{\gamma} \, \dist^p(h_{k}, U^*), \quad \bar b_k = 0
\end{align}

By Lemma~\ref{lemma-polyak11}, it follows that the sequence $\{\bar v_k\}$ converges to a non-negative scalar for any $u^*\in U^*$, and surely we have
\[
\sum_{k=N}^{\infty}  \dist^p(h_k, U^*) < \infty,\quad \sum_{k=N}^{\infty} (\|u_k - h_k\|^2 +\|u_k-h_{k-1}\|^2)< \infty.\]
Thus, it follows that 
\begin{equation}\label{eq-an0-det}
\lim_{k\to\infty}\dist^p(h_k, U^*) =0
\end{equation}
\begin{equation}\label{eq-an1-det}
\lim_{k\to\infty}\|u_k - h_k\| = 0 \end{equation}
\begin{equation}\label{eq-an4-det}
\lim_{k\to\infty}\|u_k-h_{k-1}\| =0
\end{equation}
Since $\bar v_k$ converges for any given $u^* \in U^*$ and based on equation~\eqref{eq-an4-det} we can conclude that $\|u_k - u^*\|$ converges for all $u^* \in U^*$. Therefore, the sequence $\{u_k\}$ is bounded and has accumulation points.
In view of relation~\eqref{eq-an1-det},
the sequences $\{u_k\}$ and $\{h_k\}$ have the same accumulation points.
Then, there exists a convergent subsequnce $\{u_{k_{i}}\}$, such that $\bar{u} \in U$ its limit point, i.e.,

\begin{equation}\label{eq-an3-det}
\lim_{i\to\infty} \|u_{k_{i}}-\bar u\|=0 \end{equation}
By relation~\eqref{eq-an1-det}, it follows that
\[\lim_{i\to\infty} \|h_{k_i}-\bar u\|=0\]
By continuity of the distance function $\dist(\cdot,U^*)$, from relation~\eqref{eq-an0-det} we conclude that $\dist(\bar u,U^*)=0$ , which implies that $\bar u\in U^*$ since the set $U^*$ is closed.
 Since the sequence $\{\|u_k - u^*\|^2\}$ 
converges for any $u^*\in U^*$, it follows that 
$\{\|u_k - \bar u\|^2\}$ 
converges , and by relation~\eqref{eq-an3-det} we conclude that 
$\lim_{k\to\infty}\|u_k - \bar u\|^2=0$.
\end{proof}

\textbf{Proof of Theorem~\ref{thm-Popov-a01-rates}}
\label{App-Popov-rates-01}
\begin{proof}
Since the solution set $U^*$ is closed, there exists a projection $u_k^*$ of the iterate $u_k$ on the set $U^*$, i.e., there exists a point $u_k^*\in U^*$ such that $\|u_k- u_k^*\|= \dist(u_k, U^*)$. Thus, by letting $u^*=u_k^*$ in~\eqref{eq-popov-basic-8}, we obtain for all $k\ge 1$,
\begin{equation}
    \begin{aligned}
    \label{eq-popov-a01-rate-1}
     \|u_{k+1}  - u_k^*\|^2 + \|u_{k+1} -h_k\|^2  &\leq \|u_k  - u_k^*\|^2   + \|u_k - h_{k-1}\|^2 \cr
      &- \frac{1}{2}\|u_k - h_{k}\|^2 - \frac{1}{2}\|u_k - h_{k-1}\|^2-  2 \mu \gamma_k \, \dist^p(h_{k}, U^*).
    \end{aligned}
\end{equation}

In view of $\|u_k- u_k^*\|= \dist(u_k, U^*)$ and 
$\dist(u_{k+1},U^*)\le \|u_{k+1} - u_k^*\|$, it follows that for all $k\ge1$,
\begin{equation}
    \begin{aligned}
    \label{eq-popov-a01-rate-2-0}
     \dist^2(u_{k+1}, U^*) + \|u_{k+1} -h_k\|^2  &\leq \dist^2(u_k, U^*)   + \|u_k - h_{k-1}\|^2 \cr
      &- \frac{1}{2}\|u_k - h_{k}\|^2 - \frac{1}{2}\|u_k - h_{k-1}\|^2-  2 \mu \gamma_k \, \dist^p(h_{k}, U^*).
    \end{aligned}
\end{equation}
Based on the result of Theorem~\ref{thm-Popov-a01-asymp} (b) the stepsizes are lower bounded, i.e. $\gamma_k \geq \underline{\gamma}$, then for all $k \geq 1$ we have:
\begin{equation}
    \begin{aligned}
    \label{eq-popov-a01-rate-2}
     \dist^2(u_{k+1}, U^*) + \|u_{k+1} -h_k\|^2  &\leq \dist^2(u_k, U^*)   + \|u_k - h_{k-1}\|^2 \cr
      &- \frac{1}{2}\|u_k - h_{k}\|^2 - \frac{1}{2}\|u_k - h_{k-1}\|^2-  2 \mu \underline{\gamma} \, \dist^p(h_{k}, U^*).
    \end{aligned}
\end{equation}

Next, we estimate the term $\dist^p(h_k,U^*)$ in~\eqref{eq-popov-a01-rate-2}. By the triangle inequality, we have
\[\|u_k - u^*\| \leq \|u_k - h_k\| + \|h_k - u^*\| \qquad \hbox{for all }u^*\in U^*,\]
and by taking the minimum over $u^*\in U^*$ on both sides of the preceding relation, we obtain
\begin{equation}\label{eq-popov-a01-rate-3}
\dist(u_k,U^*)\le \|u_k - h_k\| +\dist(h_k,U^*).
\end{equation}

\textbf{Case (a)}
Squaring both sides of inequality~\eqref{eq-popov-a01-rate-3}, and using $(\sum_{i=1}^m a_i)^2\le m\sum_{i=1}^m a_i^2$, which is valid for any scalars $a_i$, $i=1,\ldots,m,$ and any integer $m\ge 1$, we obtain
\begin{equation}\label{eq-popov-a01-rate-4}
- 2 \, \dist^2(h_k, U^*) 
\leq 2 \|u_k - h_k\|^2 - \, \dist^2(u_k, U^*).
\end{equation}

Combining this inequality with \eqref{eq-popov-a01-rate-2}, and adding and subtracting $\mu \underline{\gamma} \|u_k - h_{k-1}\|^2$ we obtain that for any $k \geq 1$:
\begin{equation}
    \begin{aligned}
    \label{eq-popov-a01-rate-5}
     \dist^2(u_{k+1}, U^*) + \|u_{k+1} -h_k\|^2  &\leq \dist^2(u_k, U^*)   + (1 -  \mu \underline{\gamma})\|u_k - h_{k-1}\|^2 \cr
      &- (\frac{1}{2} - 2 \mu \underline{\gamma}) \|u_k - h_{k}\|^2 - (\frac{1}{2} -  \mu \underline{\gamma})\|u_k - h_{k-1}\|^2 \cr
      &-  \mu \underline{\gamma} \, \dist^2(u_{k}, U^*).
    \end{aligned}
\end{equation}
By the stepsize choice, $\underline{\gamma} \leq \frac{1}{4 \mu}$ then $- (\frac{1}{2} - 2 \mu \underline{\gamma}) \|u_k - h_{k}\|^2 \leq 0$ and $- (\frac{1}{2} -  \mu \underline{\gamma})\|u_k - h_{k-1}\|^2 \leq 0$. Thus, by defining $R^2_{k} = \dist^2(u_k, U^*) + \|u_{k} -h_{k-1}\|^2$  we obtain that for all $k \ge 1$,
\begin{equation}
    \begin{aligned}
    \label{eq-popov-a01-rate-6}
     R^2_{k+1} & \leq (1 -  \mu \underline{\gamma}) R_k^2.
    \end{aligned}
\end{equation}

\textbf{Case (b)}
By using Lemma~\ref{inequality-hoed-2},
we further obtain
\begin{equation}
\begin{aligned}
\label{eq-popov-a01-rate-7-0}
\dist^p(u_k,U^*) &\le (\|u_k - h_k\| +\dist(h_k,U^*))^p \cr
&\le 2^{p-1} \|u_k - h_k\|^p + 2^{p-1} \, \dist^p(h_k,U^*).
\end{aligned}
\end{equation}
Using projection inequality~\eqref{eq-proj3}, and stepsizes choice  we obtain:
\[ \|u_k - h_k\| \leq \|\gamma_{k} F(h_{k-1})\| \leq  1, \]
\[ \|u_k - h_k\|^{p-2} \leq 1. \]
Combining this result with~\eqref{eq-popov-a01-rate-7-0}
\begin{equation}
\begin{aligned}
\label{eq-popov-a01-rate-7}
\dist^p(u_k,U^*) &\le  2^{p-1}  \|u_k - h_k\|^2 + 2^{p-1} \, \dist^p(h_k,U^*).
\end{aligned}
\end{equation}
And by rearranging terms we obtain:
\begin{equation}
\begin{aligned}
\label{eq-popov-a01-rate-8}
- \dist^p(h_k,U^*) &\le \|u_k - h_k\|^2 - 2^{-p+1}  \, \dist^p(u_k,U^*)
\end{aligned}
\end{equation}
Combining this inequality with \eqref{eq-popov-a01-rate-2},
we obtain that for any $k \geq 1$:
\begin{equation}
    \begin{aligned}
    \label{eq-popov-a01-rate-9-0}
     \dist^2(u_{k+1}, U^*) + \|u_{k+1} -h_k\|^2  &\leq \dist^2(u_k, U^*)   + \|u_k - h_{k-1}\|^2 - (\frac{1}{2} - 2 \mu \underline{\gamma} )\|u_k - h_{k}\|^2 \cr
      &- \frac{1}{2}\|u_k - h_{k-1}\|^2-  2^{-p+2} \mu \underline{\gamma} \, \dist^p(u_{k}, U^*).
    \end{aligned}
\end{equation}
By the stepsizes choice $\gamma_k \leq \frac{1}{4 \mu}$ for all $k \geq 1$, then $(\frac{1}{2} - 2 \mu \underline{\gamma})\|u_k - h_{k}\|^2 \leq 0$, and for all $k\geq 1$
\begin{equation}
    \begin{aligned}
    \label{eq-popov-a01-rate-9}
     \dist^2(u_{k+1}, U^*) + \|u_{k+1} -h_k\|^2  &\leq \dist^2(u_k, U^*)   + \|u_k - h_{k-1}\|^2 \cr
      &- \frac{1}{2}\|u_k - h_{k-1}\|^2-  2^{-p+2} \mu \underline{\gamma} \, \dist^p(u_{k}, U^*).
    \end{aligned}
\end{equation}
To get convergence rate results we define $R_k^2 = \dist^2(u_{k+1}, U^*) + \|u_{k+1} -h_k\|^2$.
Now we estimate $-\frac{1}{2} \|u_{k} - h_{k-1}\|^2$. From the result of Theorem~\ref{thm-Popov-a01-asymp} (a) we get that for all $k \geq 1$,  $\|u_{k} - h_{k-1}\| \leq \sqrt{2} R_1$. Then for all $k \geq 1$

\begin{equation}
    \begin{aligned}
    \label{eq-popov-a01-rate-10}
    \|u_k - h_{k-1}\|^{p-2}  \|u_k - h_{k-1}\|^2 \leq (2 R_1^2)^{(p-2)/2}  \|u_k - h_{k-1}\|^2
    \end{aligned}
\end{equation}
Then by multiplying both sides by $- \frac{1}{2}\|u_k - h_{k-1}\|^{p-2} $ we obtain:
\begin{equation}
    \begin{aligned}
    \label{eq-popov-a01-rate-11}
    - \frac{1}{2}\|u_k - h_{k-1}\|^2 \leq -\frac{1}{2}(2 R_1^2)^{-(p-2)/2}  \|u_k - h_{k-1}\|^p
    \end{aligned}
\end{equation}
Using this bound and equation~\eqref{eq-popov-a01-rate-9} we get:
\begin{equation}
    \begin{aligned}
    \label{eq-popov-a01-rate-12}
     \dist^2(u_{k+1}, U^*) + \|u_{k+1} -h_k\|^2  &\leq \dist^2(u_k, U^*)   + \|u_k - h_{k-1}\|^2  \cr
      &-\frac{1}{2}(2 R_1^2)^{-(p-2)/2}\|u_k - h_{k-1}\|^p-  2^{-p+2} \mu \underline{\gamma} \, \dist^p(u_{k}, U^*).
    \end{aligned}
\end{equation}

Let $C_4 = \min \{ \frac{1}{2}(2 R_1^2)^{-(p-2)/2} , 2^{-p+2} \mu \underline{\gamma}\}$, note that $C_4 \leq 2^{-p+2} \mu \underline{\gamma} \leq 4 \mu \underline{\gamma} \leq 1 $ when $p > 2$. Then we obtain for all $k \geq 1$
\begin{equation}
    \begin{aligned}
    \label{eq-popov-a01-rate-13}
     \dist^2(u_{k+1}, U^*) + \|u_{k+1} -h_k\|^2  &\leq \dist^2(u_k, U^*)   + \|u_k - h_{k-1}\|^2  \cr
      &-C_4 (\|u_k - h_{k-1}\|^p +  \dist^p(u_{k}, U^*)).
    \end{aligned}
\end{equation}
Finally, using ineqaulity~\eqref{inequality-hoed-2} with $a_1 = \|u_k - h_{k-1}\|^p$, $a_2 =  \dist^2(u_{k}, U^*)$ we obtain
\begin{equation}
    \begin{aligned}
    \label{eq-popov-a01-rate-14}
    - ((\|u_k - h_{k-1}\|^2)^{p/2} +  (\dist^2(u_{k}, U^*))^{p/2}) \leq 2^{-p/2+1} (\|u_k - h_{k-1}\|^2 + \dist^2(u_{k}, U^*))^{p/2}
    \end{aligned}
\end{equation}
Letting $R_k^2 = \|u_k - h_{k-1}\|^2 + \dist^2(u_{k}, U^*)$, and using the fact that $C_4 2^{-p/2+1} \leq 1$ for $p>2$, we obtain that for all $k \geq 1$:

\begin{equation}
    \begin{aligned}
    \label{eq-popov-a01-rate-15}
     R_{k+1}^2  &\leq R_k^2 - C_4 2^{-p/2+1} (R_k^2)^{p/2}.
    \end{aligned}
\end{equation}
Using Lemma~\ref{lemma-polyak-6} with $x_k = R^2_k$, $q = \frac{p-2}{2}$, $a = C_4 2^{1 - p/2}$ we obtain, for all $k\geq 1$:
\begin{equation}
    \begin{aligned}
    \label{eq-popov-a01-rate-16}
     R_{k+1}^2 & \leq \frac{R_1^2}{ (1 + (p-2) C_4 2^{-(p+2)/2}   (R_1^2)^{(p-2)/2} k)^{2/(p-2)}}.
    \end{aligned}
\end{equation}


\end{proof}

\section{Analysis of Korpelevich method with Adaptive Clipping~\ref{Adaptive-Clipping}}
\textbf{Proof of Theorem~\ref{thm-Korpelevich-c2-asym}}
\label{App-Korp-c2-asym}
\begin{proof}
\textbf{(a)}
By Lemma~\ref{Lemma1-Korp}, the  following inequality holds for any $y \in U$ and all $k\ge0$:
\begin{equation}
\begin{aligned}
\label{eq-korpelevich-c2-0}
 \|h_{k+1}  - y\|^2 &\leq \|h_k - y\|^2  - \|h_k - u_k\|^2 - 2 \g_k \langle F(u_k), u_k - y \rangle + \g_k^2 \|F(h_k) - F(u_k)\|^2.
\end{aligned}
\end{equation}
We want to estimate the last term on the LHS of the inequality using the fact that operator $F(\cdot)$ is an $\a$-symmetric operator.

\textbf{Case (I) $\a \in (0,1)$}. 

Using the alternative characterization of $\a$-symmetric operators from Proposition~\ref{prop-a}(a) (as given in ~\eqref{alpha-property}), when $\a \in (0,1)$, the next inequality holds for any $k\ge1$:
\begin{equation}
\begin{aligned}
\label{eq-korpelevich-c2-1-0}
\|F(h_k) - F(u_k)\| \leq \|h_k - u_k\| (K_0 + K_1 \|F(h_k)\|^{\alpha} + K_2 \|h_k - u_k \|^{\a / (1 - \a)}).
\end{aligned}
\end{equation}
By the projection property~\eqref{eq-proj3} and the stepsize choice~\eqref{korp-step-01} :
\[\| h_{k} - u_k\| \leq \g_k \| F(h_{k})\| = \beta_k \min\{1, \frac{1}{\|F(h_k)\|}\} \|F(h_k)\| \leq \beta_k \leq 1. \]
Then by the relation $(\sum_{i=1}^m a_i)^2\le m\sum_{i=1}^m a_i^2$ we obtain:

\begin{equation}
\begin{aligned}
\label{eq-korpelevich-c2-1-1}
\gamma_k \|F(h_k) - F(u_k)\| &\leq \gamma_k (K_0 + K_1 \|F(h_k)\|^{\alpha} + K_2 ) \|h_k - u_k\| \cr
&\leq \beta_k ( K_0 \min \{1, \frac{1}{\|F(h_{k})\|}\} +  K_1 \min \{1, \frac{1}{\|F(h_{k})\|}\} \|F(h_k)\|^{\alpha} \cr
&+  K_2 \min \{1, \frac{1}{\|F(h_{k})\|}\} ) \|h_k - u_k\| \cr
&\leq \beta_k ( K_0 +  K_1  +  K_2  ) \|h_k - u_k\| . \cr
\end{aligned}
\end{equation}


\textbf{Case (II) $\a=1$}.

Based on the alternative characterization of $\a$-symmetric operators from Proposition~\ref{prop-a}(b) (as given in ~\eqref{alpha-property-1}), when $\a = 1$, the next inequality holds for any $k\ge0$:
\begin{equation}
\begin{aligned}
\label{eq-korpelevich-c2-2-0}
\|F(h_{k}) -  F(u_k)\| &\leq \|h_k - u_k\| (L_0+ L_1 \|F(h_k)\|) \exp (L_1 \|h_{k} - u_k\|).
\end{aligned}
\end{equation}
By the projection property, triangle inequality and step size choice~\eqref{stepsizes-korp-adap}:
\[\| h_{k} - u_k\| \leq  \g_k \| F(h_k)\| = \beta_{k} \min \{1, \frac{1}{\|F(h_k)\|}\} \leq \beta_k. \]

Then, we obtain:
\begin{equation}
\begin{aligned}
\label{eq-korpelevich-c2-2-1}
\gamma_k \|F(h_{k}) -  F(u_k)\| &\leq  \gamma_k (L_0+ L_1 \|F(h_{k})\|) \exp (L_1 \beta_k) \|h_k - u_k\| \cr
&=  \exp (L_1 \beta_k) ( L_0  \beta_k \min\{1, \frac{1}{\|F(h_k)\|} \} + L_1  \beta_k\min\{1, \frac{1}{\|F(h_k)\|} \} \|F(h_{k})\|) \|h_k - u_k\| \cr
&\leq \exp (L_1 \beta_k) \beta_k ( L_0 + L_1) \|h_k - u_k\|.
\end{aligned}
\end{equation}
Now let $C_a(\beta_k) = K_0 + K_1 + K_2$ when $\alpha \in (0,1)$, and $C_a(\beta_k) = \exp (L_1 \beta_k) ( L_0 + L_1)$ when $\alpha = 1$. Then for both cases$ \g_k^2 \|F(h_{k}) -  F(u_k)\|^2 \leq \beta_k^2 C_a^2(\beta_k) \|h_k - u_k\|^2$, then combining this fact with~\eqref{eq-korpelevich-c2-0} we obtain that for any $k \ge 0$:



\begin{equation}
\begin{aligned}
\label{eq-korpelevich-c2-3}
 \|h_{k+1}  - y\|^2 &\leq \|h_k - y\|^2  - (1 - C_a^2(\beta_k))\|u_k - h_k\|^2 - 2 \g_k \langle F(u_k), u_k - y \rangle .
\end{aligned}
\end{equation}


Now we plug $y=u^*$ into equation~\eqref{eq-korpelevich-c2-3}, where $u^* \in U^*$ is an arbitrary solution. Thus, by using $p$-quasi sharpness of the operator, we obtain the following recursive inequality:
\begin{equation}
\begin{aligned}
\label{eq-korpelevich-c2-4}
 \|h_{k+1}  - u^*\|^2 &\leq \|h_k - u^*\|^2  - (1 - C_a^2(\beta_k)) \|u_k - h_k\|^2 - 2 \g_k \mu \dist^p( u_k, U^*).
\end{aligned}
\end{equation}
Now, if there exists $N > 0$ such that $ C_a^2(\beta_k) < 1$ for all $k \geq N$, then we conclude that for any $u^* \in U^*$, the sequence $\{\|h_k - u^*\|\}$ is bounded $\bar{D}_N$ by where $\bar{D}_N = \max_{k \in [0, N]} \{ \dist(h_k, U^*) \}$ for all $k \geq 0$.

\textbf{(b)} Now, we want to estimate $\|F(h_k)\|$, since this term is present in the denominator of the stepsize. Let $h^*_0 = P_{U^*}(h_0)$ be a projection of $h_0$ onto the solution set $U^*$. We add and subtract $F(h^*_0)$ , and get $\|F(h_{k})\| = \|F(h_{k}) - F(h^*_0) + F(h^*_0)\| \leq \|F(h_{k}) - F(h^*_0)\| + \|F(h^*_0)\|$. We can estimate the first term using the $\alpha$-symmetric assumption on the operator class.

\textbf{Case(I) $\alpha \in (0,1)$}

Based on the alternative characterization of $\a$-symmetric operators from Proposition~\ref{prop-a}(a) (as given in ~\eqref{alpha-property}),
\begin{equation}
\begin{aligned}
\label{eq-korpelevich-c2-5}
\|F(h_k) - F(h_0^*)\| \leq \|h_k - h_0^*\| (K_0 + K_1 \|F(h_0^*)\|^{\alpha} + K_2 \|h_k - h_0^* \|^{\a / (1 - \a)}).
\end{aligned}
\end{equation}
Earlier, in part (a) we proved that for arbitrary solution $u^*$ the following bound holds for any $k \geq 0$ 
\[\|h_{k} - u^*\|  \leq \bar{D}_N .\]
Therefore it holds for $u^*=h_0^*$,
\[\|h_{k} - h_0^*\|  \leq \bar{D}_N .\]
Using this fact and equation~\eqref{eq-korpelevich-c2-5} for we obtain that for all $k \ge 0$:

\begin{align}\label{eq-korpelevich-c2-6}
\|F(h_k) \| \leq \bar{D}_N (K_0 + K_1 \|F(h_0^*)\|^{\alpha} + K_2 \bar{D}_N^{\a / (1 - \a)}) + \|F(h_0^*)\|.
\end{align}

We showed that the sequence \{$\|F(h_k)\|$\} is upper bounded by some constant $C_1$. Where $C_1 = \bar{D}_N (K_0 + K_1 \|F(h_0^*)\|^{\alpha} + K_2 \bar{D}_N^{\a / (1 - \a)}) + \|F(h_0^*)\|$.
Using this fact, we conclude that for all $k \geq 0$
\begin{align}
\label{eq-korpelevich-c2-7}
\g_k &= \beta_k \min \{ 1 , \frac{1}{\|F(h_{k})\|}  \} \cr
&\geq \beta_k \min \{ 1, \frac{1}{C_1}\} 
\end{align}

\textbf{Case(II) $\alpha = 1$}

Based on the alternative characterization of $\a$-symmetric operators from Proposition~\ref{prop-a}(b) (as given in ~\eqref{alpha-property-1}),
\begin{align}\label{eq-korpelevich-c2-8}
\|F(h_{k}) -  F(h_0^*)\| &\leq \|h_k - h_0^*\| (L_0+ L_1 \|F(h_0^*)\|) \exp (L_1 \|h_{k} - h_0^*\|).
\end{align}
Using the preceding relation we get the following bound for $\|F(h_{k-1})\|$, for all $k \ge 1$:
\begin{align}\label{eq-korpelevich-c2-9}
\|F(h_{k})\| &\leq \|F(h_{k}) - F(h^*_0)\| + \|F(h^*_0)\| \cr
&\leq \|h_k - h_0^*\| (L_0+ L_1 \|F(h_0^*)\|) \exp (L_1 \|h_{k} - h_0^*\|)+  \|F(h^*_0)\| \cr
\end{align}
Earlier, in part(a) we proved that, for arbitrary solution $u^*\in U^*$, the following bound hold for any $k \geq 0$ 
\[\|h_{k} - u^*\| \leq  \bar{D}_N.\]

Then, it holds for $u^* = h_0^*$, and for all $k \geq 0$ it holds that
\[ \|h_{k} - h_0^*\| \leq \bar{D}_N.\]

Using this fact and equation~\eqref{eq-korpelevich-c2-9} for all $k \ge 0$:

\begin{align}\label{eq-korpelevich-c2-10}
\|F(h_{k})\| &\leq  \|h_k - h_0^*\| (L_0+ L_1 \|F(h_0^*)\|) \exp (L_1 \|h_{k} - h_0^*\|) + \|F(h_0^*)\| \cr
&\leq \bar{D}_N (L_0+ L_1 \|F(h_0^*)\|) \exp (L_1 \bar{D}_N) +  \|F(h_0^*)\| .
\end{align}
We showed that the sequence \{$\|F(h_k)\|$\} is upper bounded by some constant $\overline{C}_1$. Where $\overline{C}_1 = \bar{D}_N (L_0+ L_1 \|F(h_0^*)\|) \exp (L_1 \bar{D}_N) +  \|F(h_0^*)\| $. Then, we conclude that 
\begin{equation}
\begin{aligned}
\label{eq-korpelevich-c2-11}
\g_k &= \beta_k \min \{ 1, \frac{1}{\|F(h_{k})\| } \} \cr
&\geq \beta_k \min \{1 , \frac{1}{ \overline{C}_1 } \}. 
\end{aligned}
\end{equation}

For both cases $\a \in (0,1)$ and $\alpha=1$ in equations~\eqref{eq-korpelevich-c2-6}, ~\eqref{eq-korpelevich-c2-11} we showed that stepsizes sequence $\{\gamma_k\}$ is lower bounded by some constant $\beta_k \min \{1, \frac{1}{C_1}, \frac{1}{\overline{C}_1}\}$.

\textbf{(c)} Using the results from part (b) and equation~\eqref{eq-korpelevich-c2-4}, we obtain for any $k\geq 0$:

\begin{equation}
\begin{aligned}
\label{eq-korpelevich-c2-11-2}
 \|h_{k+1}  - u^*\|^2 &\leq \|h_k - u^*\|^2  -\frac{1}{2}\|u_k - h_k\|^2 - 2 \beta_k \mu \min \{1, \frac{1}{C_1}, \frac{1}{\overline{C}_1}\} \, \dist^p(u_k, U^*).
\end{aligned}
\end{equation}

The equation~\eqref{eq-korpelevich-c2-11-2} satisfies the condition of Lemma~\ref{lemma-polyak11-det} with 
\begin{align}\bar v_{k} = \|h_{k}  - h^*\|^2, \quad  \bar a_k = 0, \cr
\bar z_k =  \frac{1}{2}\|u_k - h_{k}\|^2 +  2 \mu \beta_k \min \{1, \frac{1}{C_1}, \frac{1}{\overline{C}_1}\} \, \dist^p(u_{k}, U^*), \quad \bar b_k = 0.
\end{align}

By Lemma~\ref{lemma-polyak11}, it follows that the sequence $\{\bar v_k\}$ converges to a non-negative scalar for any $u^*\in U^*$, 
we have
\[
\sum_{k=0}^{\infty}  \beta_k \dist^p(u_k, U^*) < \infty,\quad \sum_{k=0}^{\infty} \|u_k - h_k\|^2 < \infty.\]
Thus, it follows that 
\begin{equation}\label{eq-korp-c2-an0-det-a1}
\lim_{k\to\infty}\beta_k \dist^p(u_k, U^*) =0,
\end{equation}
and since $\sum_{k = 0}^{\infty} \beta_k = \infty$,
\begin{equation}\label{eq-korp-c2-an0-det-a1-1}
\liminf_{k\to\infty} \dist^p(u_k, U^*) =0,
\end{equation}
\begin{equation}\label{eq-korp-c2-an1-det-a1}
\lim_{k\to\infty}\|u_k - h_k\| = 0. \end{equation}

Since $\bar v_k$ converges for any given $u^* \in U^*$  we can conclude that $\|h_k - u^*\|$ converges for all $u^* \in U^*$. Therefore, the sequence $\{h_k\}$ is bounded and has accumulation points.
In view of relation~\eqref{eq-korp-c2-an1-det-a1},
the sequences $\{u_k\}$ and $\{h_k\}$ have the same accumulation points. Let $\{k_i\}$ be an index sequence, such that

\begin{equation}\label{eq-korp-c2-an1-det-a1-1}
\lim_{i \to\infty} \dist^p(u_{k_i}, U^*) = \liminf_{k\to\infty} \dist^p(u_k, U^*) =0,
\end{equation}

We assume that the sequence $\{u_{k_i}\}$ is convergent with a limit point $\bar{u}$, otherwise, we choose a convergent subsequence:
\begin{equation}\label{eq-korp-c2-an3-det-a1}
\lim_{i\to\infty} \|u_{k_i}-\bar u\|=0 \end{equation}
By relation~\eqref{eq-korp-c2-an1-det-a1}, it follows that
\begin{equation} \label{eq-korp-c2-an4-det-a1}
\lim_{i\to\infty} \|h_{k_i}-\bar u\|=0 \end{equation}
By continuity of the distance function $\dist(\cdot,U^*)$, from relation~\eqref{eq-korp-c2-an0-det-a1} we conclude that $\dist(\bar u,U^*)=0$ , which implies that $\bar u\in U^*$ since the set $U^*$ is closed.
 Since the sequence $\{\|h_k - u^*\|^2\}$ 
converges for any $u^*\in U^*$, it follows that 
$\{\|h_k - \bar u\|^2\}$ 
converges , and by relation~\eqref{eq-korp-c2-an4-det-a1} we conclude that 
$\lim_{k\to\infty}\|h_k - \bar u\|^2=0$.

\end{proof}

\end{document}